\documentclass[10pt,a4paper]{article}		

\usepackage[utf8]{inputenc}					

\usepackage{lmodern}						

\usepackage[auth-sc]{authblk}						

\usepackage{amsmath,amsfonts,amssymb,amsthm,pifont}
\usepackage{amsbsy,amscd,mathrsfs,graphicx}

\usepackage{tikz}
\usetikzlibrary{shapes,arrows}
	
\usepackage[left=3cm,right=3cm,top=3cm,bottom=3cm]{geometry}

\usepackage[makeroom]{cancel}				
\usepackage{mathtools}						
\mathtoolsset{showonlyrefs,showmanualtags}	

\theoremstyle{plain}
\newtheorem{theorem}{Theorem}
\newtheorem{prop}[theorem]{Proposition}
\newtheorem{cor}[theorem]{Corollary}
\newtheorem{lemma}[theorem]{Lemma}
\newtheorem*{lemma*}{Lemma}

\theoremstyle{definition}
\newtheorem{definition}[theorem]{Definition}

\theoremstyle{remark}
\newtheorem{rmk}{Remark}
\newtheorem{example}{Example}

\DeclareMathOperator{\spn}{\mathrm{span}}
\DeclareMathOperator{\rank}{\mathrm{rank}}
\DeclareMathOperator{\tr}{\mathrm{Tr}}
\let\div\relax\DeclareMathOperator{\div}{\mathrm{div}}

\DeclareMathOperator{\grad}{\mathrm{grad}}

\newcommand{\R}{\mathbb{R}}						
\newcommand{\distr}{\blacktriangle}				
\newcommand{\V}{\mathbf{c}}						
\newcommand{\ad}{\mathrm{ad}}					
\newcommand{\popp}{\mathcal{P}}					
\newcommand{\Z}{Z}								
\newcommand{\metr}{\mathbf{g}}					
\newcommand{\g}{\mathbf{g}}						
\newcommand{\gr}{\mathbf{gr}}
\renewcommand{\hat}{\widehat}

\newcommand{\bqn}{\begin{eqnarray}}
\newcommand{\eqn}{\end{eqnarray}}
\newcommand{\bi}{\begin{itemize}}

\newcommand{\ei}{\end{itemize}}
\newcommand{\bc}{\begin{cor}}
\newcommand{\ec}{\end{cor}}
\newcommand{\bp}{\begin{prop}}
\newcommand{\ep}{\end{prop}}
\newcommand{\bt}{\begin{theorem}}
\newcommand{\et}{\end{theorem}}

\newcommand{\eps}{\varepsilon}

\newcommand{\hh}{h}

\newcommand{\cyl}{
\begin{tikzpicture}
\tikzset{every node/.style={cylinder, shape border rotate=90, draw}}
\node [minimum height=9pt, minimum width = 9pt, aspect =.3, scale=0.8, very thin] at (0,0) {};
\end{tikzpicture}
}

\newcommand{\lp}{\left(}
\newcommand{\rp}{\right)}
\newcommand{\lb}{\left[}
\newcommand{\rb}{\right]}
\newcommand{\lc}{\left\{}
\newcommand{\rc}{\right\}}
\newcommand{\lab}{\left|}
\newcommand{\rab}{\right|}

\newcommand{\E}{\mathbb{E}}
\newcommand{\bR}{\mathbb{R}}

\newcommand{\sM}{\mathcal{M}}

\newcommand{\bS}{\mathbb{S}}

\title{Intrinsic random walks and sub-Laplacians in sub-Riemannian geometry}

\author[1]{Ugo Boscain}
\author[2]{Robert Neel}
\author[3]{Luca Rizzi}

\affil[1]{CNRS, CMAP \'Ecole Polytechnique, and \'Equipe INRIA GECO Saclay \^Ile-de-France, Paris}
\affil[2]{Department of Mathematics, Lehigh University, Bethlehem, PA, USA}
\affil[3]{Univ.\ Grenoble Alpes, CNRS, Institut Fourier, F-38000 Grenoble, France}

\date{\today}

\begin{document}

\maketitle

\begin{abstract}

On a sub-Riemannian manifold we define two type of Laplacians. The \emph{macroscopic Laplacian} $\Delta_\omega$, as the divergence of the horizontal gradient, once a volume $\omega$ is fixed, and the \emph{microscopic Laplacian}, as the operator associated with a sequence of geodesic random walks. We consider a general class of random walks, where \emph{all} sub-Riemannian geodesics are taken in account. This operator depends only on the choice of a complement $\V$ to the sub-Riemannian distribution, and is denoted $L^\V$.

We address the problem of equivalence of the two operators. This problem is interesting since, on equiregular sub-Riemannian manifolds, there is always an intrinsic volume (e.g. Popp's one $\popp$) but not a canonical choice of complement. The result depends heavily on the type of structure under investigation:

\begin{itemize}
\item On contact structures, for every volume $\omega$, there exists a unique complement $\V$ such that $\Delta_\omega=L^\V$.
\item On Carnot groups, if $H$ is the Haar volume, then there always exists a complement $\V$ such that $\Delta_H=L^\V$. However this complement is \emph{not unique} in general.
\item For quasi-contact structures, in general, $\Delta_\popp\neq L^\V$ for any choice of $\V$. In particular, $L^\V$ is not symmetric with respect to Popp's measure. This is surprising especially in dimension $4$ where, in a suitable sense, $\Delta_\popp$ is the unique intrinsic macroscopic Laplacian.
\end{itemize}
A crucial notion that we introduce here is the \emph{N-intrinsic volume}, i.e.\ a volume that depends only on the set of parameters of the nilpotent approximation. When the nilpotent approximation does not depend on the point, a N-intrinsic volume is unique up to a scaling by a constant and the corresponding N-intrinsic sub-Laplacian is unique. This is what happens for dimension less than or equal to $4$, and in particular in the 4-dimensional quasi-contact structure mentioned above.

Finally, we prove a general theorem on the convergence of families of random walks to a diffusion, that gives, in particular, the convergence of the random walks mentioned above to the diffusion generated by $L^\V$.
\end{abstract}

\tableofcontents

\section{Introduction}\label{s:intro}

\subsection{The Riemannian setting}\label{RiemIntro}

Let $M$ be a (smooth, connected, orientable,  complete) $n$-dimensional  Riemannian manifold. In Riemannian geometry  the infinitesimal  conservation condition for a smooth scalar quantity $\phi$, a function of the point $q$ and of the time $t$ (for example, the temperature, the concentration of a diffusing substance, the noise, the probability density of a randomly-moving particle, etc.), that flows via a flux $F$
(which says how much of $\phi$ is, infinitesimally, crossing a unit area of surface and in a unit of time) is expressed via the ``continuity'' equation $\partial_t\phi +\div(F)=0$, where $\div(\cdot)$ is the divergence computed with respect to the Riemannian volume. If one postulates that the flux is proportional to minus the Riemannian gradient (the constant of proportionality being fixed to 1 for simplicity), that is $F=-\grad(\phi)$,  one obtains the Riemannian heat equation
\begin{equation}\label{heat-1}
\partial_t\phi =\Delta\phi,
\end{equation}
where $\Delta = \div \circ \grad$ is the Laplace-Beltrami operator.
Since equation~\eqref{heat-1} has been obtained from a continuity equation that views $\phi$ as a fluid without a microscopic structure (the fluid is modeled as a continuous substance), in the following we refer to it as to the {\em macroscopic heat equation} and to the corresponding operator $\Delta$ as the {\em macroscopic Laplacian}. It is useful to write $\Delta$ in terms of an orthonormal frame. If $ X_1,\ldots, X_n $ is a local orthonormal frame for the Riemannian structure we have the formulas
\begin{equation}
\grad(\phi)=\sum_{i=1}^n X_i(\phi) X_i \qquad \text{and} \qquad \Delta(\phi)=\sum_{i=1}^n (X_i^2 +\div(X_i)X_i   )(\phi).
\end{equation}
It is well known that the heat equation (\ref{heat-1})  also admits a stochastic interpretation as the evolution equation associated to a diffusion process. Actually, we need to be more precise here, since there are two evolution equations associated to a diffusion process. In particular, there is the evolution of the expectation of a test function $f\in C_0^{\infty}(M)$ evaluated along the paths of the diffusion, called the forward Kolmogorov equation, and there is the evolution equation for the transition density of the diffusion (relative to a smooth volume), called the backward Kolmogorov equation. The second-order operators in these two equations are adjoints (with respect to the same smooth volume as the transition density). However, thanks to the geodesic completeness assumption, the  Laplace-Beltrami operator is essentially self-adjoint (with respect to the Riemannian volume) so that both operators, and thus both parabolic PDEs, coincide.

Moreover, the diffusion process associated to the (Riemannian) heat equation can be obtained as a (parabolic) scaling limit of a random walk. We think of the random walk as giving a ``microscopic'' viewpoint on the evolution of the quantity measured by $\phi$, since, at least in an idealized sense, it models the motion of an individual particle. In taking the scaling limit, we pass from the small-scale behavior of an individual particle to the large-scale average behavior of the particle, or equivalently, to the aggregate  behavior of a large number of such particles, and thus pass from the microscopic to the macroscopic viewpoint. We now need to explain these random walks and the associated ideas in more detail.

In the Riemannian case currently under discussion, we are interested in an isotropic random walk. In particular, starting from a point $q$, the unit sphere $\mathbb{S}^{n-1}$ in the tangent space $T_{q}M$ has the standard $(n-1)$-volume induced by the metric on $T_{q}M$. Normalizing this to a probability measure $\mu_{q}$ (by dividing by the total volume $2\pi^{n/2}\Gamma(n/2)$), we can then choose a direction $\theta\in\bS^{n-1}$ at $q$ randomly, in a way which is obviously isotropic with respect to the Riemannian structure. The particle then travels along the geodesic tangent to $\theta$ for a distance $\eps$ in time $\delta t$ (at constant speed, determined by these conditions). We let $X^{\eps}_t$ be the (random) position of the particle at time $t\in[0,\delta t]$. We can continue this process by next choosing a direction, at the point $X^{\eps}_{\delta t}$, at random via the measure $\mu_{X^{\eps}_{\delta t}}$ on the unit sphere $\bS^{n-1}\subset T_{X^{\eps}_{\delta t}}M$, and independently from the previous choice of direction $q$. Then the particle follows the geodesic from $X^{\eps}_{\delta t}$ in this direction for distance $\eps$ in time $\delta t$. We can continue this process indefinitely, since after $i$ steps, the particle has traveled along a piecewise geodesic for a total distance of $i\eps$, and  since we are assuming that $M$ is geodesically complete, a next step can always be made (equivalently, this process cannot explode by exiting every compact set in finite time). The result is a random (piecewise geodesic) path $X^{\eps}_t$ (for $t\in[0,\infty)$) starting from $q=X^{\eps}_0$. Further, the positions of the particle at the times when it randomly changes direction, namely, the sequence $X^{\eps}_0, X^{\eps}_{\delta t},X^{\eps}_{2\delta t},\ldots$ is a Markov chain on $M$, and for $t\in(i\delta t,(i+1)\delta t)$, $X^{\eps}_t$ interpolates between $X_{i\delta t}$ and $X_{(i+1)\delta t}$ along a geodesic between them (for $i=0,1,2,\ldots$).

As mentioned, we are interested in a scaling limit of such random walks (and this is why we have already indexed our walk $X^{\eps}_t$ by the step-size $\eps$). We will take $\delta t= \eps^2/\alpha$, where $\alpha$ is a normalizing constant to be chosen later. Then $\delta t\rightarrow 0$ as $\eps\rightarrow 0$, and moreover, this is the parabolic scaling (which we note is in a sense responsible for the infinite velocity of propagation of the heat).

We are interested in the behavior of a single step under this parabolic scaling, which by the Markov property and homogeneity in time may as well be the first step. If we consider the change in the expectation of a function $\phi\in C_0^{\infty}(M)$ sampled after the step, normalized by dividing by $\delta t$, we obtain an operator which we denote by $L_{\eps}$. More concretely, we have
\begin{equation}
\begin{split}
(L_{\eps} \phi)(q) &= \frac{ \E\lb \left . \phi\lp X^{\eps}_{\delta t}\rp \right| X^{\eps}_0=q\rb -\phi\lp q\rp}{\delta t} \\
&= \frac{\alpha}{\eps^2}\lp \E\lb \left . \phi\lp X^{\eps}_{\eps^2/\alpha}\rp \right| X^{\eps}_0=q\rb -\phi\lp q\rp \rp \\
&=  \frac{\alpha}{\eps^2} \lp \int_{\bS^{n-1}}\phi\lp\exp_{q} \lp\eps,\theta\rp\rp \mu_{q}\lp \theta\rp -\phi(q)\rp ,
\end{split}
\end{equation}
where $\exp_{q}(\eps,\theta)$ is the point obtained by following the arc-length parametrized geodesic from $q$ in the direction of $\theta$ for distance $\eps$.

It turns out that the limiting behavior of the sequence of random walks, as $\eps\rightarrow 0$, is governed by the limiting behavior of these operators. Indeed, Section \ref{a:randomwalk} is devoted to a discussion of this issue and to the proof of theorems on the convergence of a sequence of random walks to a diffusion in the Riemannian or sub-Riemannian context (see Theorem~\ref{AppendixConvergence} for a general convergence result and Theorem~\ref{AppendixConvergence2} for the case of random walks of the type just described) . At any rate, expanding the exponential map in normal coordinates around $q$ shows that $L_{\eps}$ converges to an operator $L$ as $\eps\rightarrow 0$ (in the sense that $L_{\eps}\phi\rightarrow L\phi$ uniformly on compacts for $\phi\in C_0^{\infty}(M)$), and that $L=\frac{\alpha}{2n}\Delta$ (the computation is a special case of the (sub)-Riemannian computations below, so we don't reproduce it). It follows that the sequence of random walks converges to the diffusion generated by $L$ (or the diffusion that solves the martingale problem for $L$), which we denote $X^0_t$. Thus, if $\phi$ is a smooth scalar quantity depending on a point $q$ and a time $t$ such that its evolution is governed infinitesimally by the dynamics of the random walk (in the sense that $\partial_t \phi(q,t)=\lim_{\delta\rightarrow0}\frac{1}{\delta}\E\lb \phi\lp X^0_{\delta},t\rp -\phi(q)|X^0_0=q\rb$), we obtain the equation
\bqn
\partial_t \phi=L\phi,
\label{eq-heat-2}
\eqn
where
\bqn
L\phi(q,t)=\lim_{\eps\to0}  \frac{\alpha}{\eps^2} \lp\int_{\bS^{n-1}}\phi\lp\exp_{q} (\eps,\theta),t\rp  \mu_{q}\lp \theta\rp -\phi(q,t)\rp.
\eqn
In the following we refer to the equation (\ref{eq-heat-2}) as to the ``microscopic heat equation'' and to  $L$ as  to the ``microscopic  Laplacian''. As we've already seen, if we take $\alpha=2n$, then we have
\begin{equation}\label{eq-Delta=L}
\Delta=L.
\end{equation}
(As an aside, the need for the normalizing constant $\alpha$ to grow with the dimension $n$ is a manifestation of the concentration of measure on the sphere.) Hence the macroscopic diffusion equation and the microscopic one coincide.

Further, define the heat kernel $p_t(q,q_0)$ as the density of the random variable $\left. X^0_t\right| X^0_0=q_0$ with respect to the Riemannian volume $\mathcal{R}$.
Said more analytically, $p_t$ is the fundamental solution to Equation \eqref{eq-heat-2}, so that
\begin{equation}
\phi(q_0,t) = \int_M \phi(q) p_t(q,q_0)  \mathcal{R}(q) = \E\lb \phi\lp X^0_t\rp -\phi(q_0) |X^0_0=q_0 \rb
\end{equation}
solves the Cauchy problem for Equation \eqref{eq-heat-2} with initial condition $\phi(q_0,0)=\phi(q_0)$. Then because $L$ is  essentially self-adjoint with respect to the Riemannian volume (it's just a scalar multiple of $\Delta$), $p_t(\cdot,q_0)$ also satisfies the microscopic heat equation \ref{eq-heat-2}. Thus $p_t$, which measures the probability density of the random paths themselves, rather than a quantity that is sampled along them, can be understood in terms of the same equations.

All of this can be viewed in the following way. On one side, the microscopic perspective is a good interpretation of the macroscopic heat equation. On the other side, the microscopic Laplacian $L$ is a good operator because it is essentially self-adjoint with respect to a volume (the Riemannian one).  This is due to the fact that it is symmetric since it can be written in divergence form thanks to Equation \eqref{eq-Delta=L} and to the geodesic completeness of the manifold.

The essential self-adjointness of $L$ with respect to some volume $\omega$, even if not necessary to define the corresponding process, is important, both because it means that the heat kernel satisfies the same equation and because it permits one to study the evolution equation \eqref{eq-heat-2} in $L^2(M,\omega)$. See Theorem~\ref{t:selfad} below.

\begin{rmk}
The Riemannian volume $\mathcal{R}$ can be defined equivalently by $\mathcal{R}(X_1,\ldots,X_n)=1$ for any oriented local orthonormal frame $X_1,\ldots,X_n$, or as the $n$-dimensional Hausdorff or spherical-Hausdorff volume (up to a constant). The above construction gives an alternative to characterize the Riemannian volume: it is the unique volume (up to constant rescaling) such that the microscopic Laplacian can be written in divergence form. These fact are much less trivial in the sub-Riemannian context. Indeed in sub-Riemannian geometry there are several notions of intrinsic volumes and the definition of the microscopic Laplacian requires   additional structure.
\end{rmk}

\subsection{The sub-Riemannian setting}

In this paper, a sub-Riemannian structure is a triple $(M,\distr,\metr)$, where $M$ is a $n$-dimensional differentiable manifold, $\distr$ is a smooth  distribution of constant rank $k<n$ satisfying the H\"ormander condition and $\metr$ is a Riemannian metric on $\distr$. Locally the structure can be assigned by an orthonormal frame $X_1,\ldots,X_k \in \Gamma(\distr)$. Here, $\Gamma(E)$ denotes the $C^\infty(M)$-module of smooth sections of any vector bundle $E$ over $M$. 

With the term (sub)-Riemannian, we mean structures that can be also Riemannian, i.e. defined as above but with $k\leq n$. Riemannian manifolds are by definition equiregular. All the information about the structure is contained in the Hamiltonian function $H: T^*M \to \R$ defined by
\begin{equation}
H(\lambda) = \frac{1}{2}\sum_{i=1}^k \langle \lambda,X_i \rangle^2, \qquad \lambda \in T^*M,
\end{equation}
where $\langle \lambda, \cdot\rangle$ denotes the action of covectors on vectors. It is a well-known fact in (sub)-Riemannian geometry that integral lines $\lambda(\cdot)$ of the Hamiltonian flow defined by $H$,
laying on the level set $H=1/2$,  project to smooth curves $\gamma(t):=\pi(\lambda(t))$ on $M$ that are arc-length parametrized geodesics, i.e. $ \|\dot{\gamma}(t)\|=1$ and, for every sufficient small interval $[t_1,t_2]$, the restriction $\gamma|_{[t_1,t_2]}$ is a minimizer of the \emph{sub-Riemannian length} $\ell(\gamma) = \int_{t_1}^{t_2} \|\dot{\gamma}(t)\| dt$ (with fixed endpoints). Here $\|\cdot\|$ denotes the norm induced by $\metr$ on $\distr$.
In Riemannian geometry these curves are precisely all the arc-length parameterized  geodesics of the structure. In the sub-Riemannian case these are called \emph{normal geodesics}. There is also another class of geodesics called \emph{abnormal geodesics}, that may not follow the above dynamic.

\subsubsection{The sub-Riemannian macroscopic Laplacian}

 As we will see later, in sub-Riemannian geometry the definition of an intrinsic volume, playing the role of the Riemannian volume, is a subtle question.  For the moment, let us assume that a  volume $\omega$ is fixed on the sub-Riemannian manifold. In this case one can write the macroscopic heat equation similarly to the Riemannian case. The only difference is that one should postulate that the flux is proportional to the horizontal gradient.  The horizontal gradient $\grad_H(\cdot)$ of a $C^\infty$ function $\phi$ is defined similarly to the Riemannian gradient but is a vector field belonging to the distribution (see, for instance, \cite{laplacian}):
\begin{equation}
\metr_q(v,\grad_H(\phi)_q)=d_q\phi(v), \qquad \forall v\in\distr_q.
\end{equation}
We have then for the macroscopic heat equation
\begin{equation}
\partial_t\phi=\Delta_\omega \phi,
\end{equation}
where 
\begin{equation}\label{eq-LB-H}
\Delta_\omega =\div_\omega\circ \grad_H
\end{equation}
is the macroscopic Laplacian.
In terms of a local orthonormal frame of the sub-Riemannian manifold we have the same formulas as in the Riemannian case, but summing up to the rank of the distribution:
\begin{equation}
\grad_H(\phi)=\sum_{i=1}^k X_i(\phi) X_i, \qquad \Delta_\omega(\phi)=\sum_{i=1}^k (X_i^2 +\div_\omega(X_i)X_i   )(\phi).
\end{equation}
Since $\grad_H$ coincides with the standard gradient in the Riemannian case, we suppress the $H$ from the notation.

\subsubsection{On intrinsic volumes}

The difficulty concerning the definition of a macroscopic Laplacian is related to the choice of $\omega$. What is the ``correct'' volume in the sub-Riemannian case, the analogue to the Riemannian one? One needs some algorithm to assign, with any sub-Riemannian structure on $M$, a volume form $\omega_M$. Moreover, the construction of $\omega_M$ should depend, loosely speaking, only on the metric invariants of the structure.

\begin{definition}
An \emph{intrinsic definition of volume} is a map that associates, with any (oriented) (sub)-Riemannian structure $(M,\distr,\g)$ a volume form $\omega_M$ on $M$ such that if $\phi: M \to N$ is a sub-Riemannian isometry between $(M,\distr_M,\g_M)$ and $(N,\distr_N,\g_N)$, then $\phi^* \mu_N  = \mu_M$.
\end{definition}
\begin{rmk}
In order to avoid this verbose terminology, with the term ``intrinsic volume'' we mean either the actual volume form $\omega_M$ or the definition of volume given by a map $M \mapsto \omega_M$.
\end{rmk}
Even in the Riemannian case, there are many intrinsic volumes. The classical Riemannian one is the unique volume form $\mathcal{R}$ such that $\mathcal{R}(X_1,\ldots,X_n)=1$ for any (oriented) orthonormal frame. But what about the volume form defined by $\mathcal{R}'(X_1,\ldots,X_n) = 1+\kappa^2$, where $\kappa$ is the scalar curvature? Both are perfectly fine definitions of volume, according to our definition. The first, loosely speaking, is more ``intrinsic'' than the second. In fact, it is true that both depend only on the metric invariant  of the structure, but $\mathcal{R}'$ involves second-order information about the structure. To rule out $\mathcal{R}'$ we need a more precise definition.

Roughly speaking we say that an intrinsic definition of volume is \emph{N-intrinsic} if its value of $\omega_M$ at $q$ depends only on the metric invariants of the nilpotent approximation of the structure at $q$ (the metric tangent space to the structure). For the precise definition see Section \ref{s-intrinsic}. In the Riemannian case, there is only one nilpotent approximation, which is the flat $\R^n$, and it has no non-trivial metric invariants. As a consequence, there is a unique N-intrinsic volume, the Riemannian one. As is well known, in the sub-Riemannian case, nilpotent approximations may be different at different points. 
\begin{definition}
A (sub)-Riemannian manifold is \emph{equi-nilpotentizable} if nilpotent approximations at different points are isometric.
\end{definition}
As just said above, Riemannian manifolds are equi-nilpotentizable. Equi-nilpotentizable structures are equiregular (see Section~\ref{s:prel}). By definition, all Carnot groups are equi-nilpotentizable.
 It is well known that all equiregular sub-Riemannian structures in dimension less than or equal to four are equi-nilpotentizable (see \cite{ABB-Hausdorff}). In particular this is true for 3D sub-Riemannian manifolds (the nilpotent approximation is given by the Heisenberg group $\mathbb{H\\
 }_3$ at each point). The simplest non-equi-nilpotentizable sub-Riemannian structure is given by a generic structure of rank $4$ in dimension $5$. Other examples of non-equi-nilpotentizable  sub-Riemannian manifolds are given by generic contact manifolds of dimension greater than or equal to 5 and by generic quasi-contact sub-Riemannian manifolds in dimension greater than or equal to 6.

For equiregular sub-Riemannian structures a general N-intrinsic volume, called \emph{Popp volume} and denoted with $\popp$, was defined in Montgomery's book \cite{montgomerybook} and later studied in \cite{laplacian,nostropopp,ABB-Hausdorff}. This definition was pioneered by Brockett for contact structures in \cite{brockett}. It turns out that, in the Riemannian case, Popp's construction recovers the Riemannian volume. Moreover, in \cite{ABB-Hausdorff}, it has been proved that, for non-equinilpotentizable structures, Popp volume does not coincide with either the Hausdorff or spherical Hausdorff volume. Finally, in Section~\ref{s-intrinsic} we prove the following result.
\begin{prop}
Let  $(M,\distr,\metr)$ be an  equi-nilpotentizable (sub)-Riemannian manifold. Then Popp volume $\popp$ is the unique N-intrinsic definition of volume, up to a multiplicative constant.
\end{prop}
 In this sense Popp volume generalizes the Riemannian one. When $\omega$ is an N-intrinsic volume, we say that $\Delta_\omega$ is an \emph{N-intrinsic macroscopic Laplacian}. Since the divergence of a vector field does not change if the volume is multiplied by a constant, we have the following.
\begin{prop}
Let $(M,\distr,\metr)$ be an equi-nilpotentizable (sub)-Riemannian manifold. Then there exists a unique N-intrinsic macroscopic Laplacian, $\Delta_\popp$ that is built with Popp's volume.
\end{prop}

\subsubsection{The sub-Riemannian microscopic Laplacian}

Another problem in the sub-Riemannian case is the definition of microscopic Laplacian, as a limit process coming from a random walk along geodesics starting from a given point. Indeed, while in the Riemannian context the geodesics starting from a given point can always be parametrized by the direction of the initial velocity, i.e., by the points of a $(n-1)$-dimensional sphere, in the sub-Riemannian context the geodesics starting from a given point are always parameterized by a non-compact set, namely by the points of a cylinder $\cyl_{q}=H^{-1}(1/2)\cap T_{q}^\ast M$ having the topology of $\mathbb{S}^{k-1}\times \R^{n-k}$. How to define an intrinsic finite volume on $\cyl_{q}$,  and thus a probability measure, is a non-trivial question. 

Moreover, while in Riemannian geometry, for sufficiently small $\eps$ the Riemannian metric sphere of radius $\eps$ centered at $q$ coincides with the endpoints of the geodesics starting from $q$ and having length $\eps$, in sub-Riemannian geometry this is not true: for a fixed a point  $q$, there are geodesics starting from $q$ that lose optimality arbitrarily close to $q$. Hence one has to decide whether to average only on the sub-Riemannian sphere (i.e., on geodesics that are optimal up to length $\eps$) or on the sub-Riemannian front (i.e., on all geodesics of length $\eps$).
Another approach is to choose a $k$-dimensional linear subspace $\mathbf{h}_q$ in $T^*_{q}M$ transverse to the non compact directions of the cylinder (playing the role of the ``most horizontal'' geodesics) and to average on $\mathbf{h}_q\cap\cyl_{q}$. This last approach is essentially the one followed in \cite{GordLae,GordLaeOlder,grong1,grong2}. Certainly the problem of finding a canonical subspace $\mathbf{h}_q$ in $T^\ast_{q}M$ is a non-trivial question and in general it is believed that it is not possible.

All these problems are encompassed in the specification of a measure $\mu_q$ (possibly singular) on $\cyl_{q}$ for all $q \in M$. Once such a collection of measures $\mu = \{\mu_{q}\}_{q \in M}$ is fixed, we define the \emph{microscopic Laplacian} as
\begin{equation}\label{eq-microscopic}
(L_{\mu}\phi)(q):=\lim_{\eps\to0}  
\frac{\alpha}{\eps^2} \Big(\int_{\cyl_{q}}\phi\big(\exp_{q} (\eps,\lambda)\big) \mu_{q}(\lambda)-\phi(q)\Big) ,
\end{equation}
 assuming the limit exists.
Here $t \mapsto \exp_{q} (t,\lambda)$ is the arc-length parametrized normal geodesics starting from $q$ with initial covector $\lambda\in\cyl_{q}$. Moreover $\alpha$ is a constant of proportionality that will be fixed later. 

\begin{rmk}\label{abnormal}
In this paper, the microscopic Laplacian is built only with the geodesics of the exponential map that, by definition, are normal.
In sub-Riemannian geometry there is also another type of geodesic, called (strict) \emph{abnormal}, that are not described by the sub-Riemannian exponential map (i.e. they are not the projection of the Hamiltonian flow described above). The ``size'' of the set of points reached by abnormal geodesics (starting from a fixed origin) is a hard open problem known as the \emph{Sard conjecture in sub-Riemannian geometry}, see for instance \cite{open,riffordbook}). It is believed that only a set of measure zero is reached by abnormal geodesics. What is known in general is that the set of points reached by optimal normal geodesics is open and dense (see \cite{smoothness,curvature} for precise statements). We notice that, on contact structures, there are no nontrivial abnormal geodesics.
 
If one would like to include also abnormal geodesics for the construction of the microscopic Laplacian, one should decide which measure give to them and in principle one could get a different operator. This research direction is beyond the purpose of the present paper.
\end{rmk}

In view of Theorem~\ref{t-lmu}, we restrict the class of possible measures, and we consider only those induced by a complement as follows. For all $q \in M$, consider a complement $\V_q$ such that $T_qM = \distr_q \oplus \V_q$. By duality $T_q^*M = \mathbf{v}_q \oplus \mathbf{h}_q$, where $\mathbf{v}_q:= \distr_q^\perp$ (resp. $\mathbf{h}_q:=\V^\perp$) denote the annihilators of $\distr_q$ (resp. $\V_q$). We can see $\mathbf{v}_q$ as the space of ``vertical'' covectors, and $\mathbf{h}_q$ the space of ``horizontal'' ones.  Now we can define a Euclidean structure on $\mathbf{h}_q$ by identifying it with $\distr_q$ (this is equivalent to considering the restriction $2H|_{\mathbf{h}_q}$, which is a positive definite quadratic form). The intersection $\mathbb{S}_q^{k-1} = \cyl_q \cap \mathbf{h}_q$ (See Figure~\ref{f:verhor} in Section~\ref{s-luca-formula}) is precisely the Euclidean sphere in this space. Then the cylinder of initial covectors splits as
\begin{equation}
\cyl_q = \mathbb{S}_q^{k-1}\times \mathbf{v}_q.
\end{equation}
We stress that this identification depends on the choice of $\V$. We restrict to the class of product measures induced by the choice of a complement on $\cyl_q$ of the form
\begin{equation}\label{eq:product}
\mu_{\V_q} =  \mu_{\mathbb{S}^{k-1}_q} \times \mu_{\mathbf{v}_q},
\end{equation}
where $\mu_{\mathbb{S}^{k-1}_q}$ is the Euclidean probability measure on $\mathbb{S}_q^{k-1} = \cyl_q \cap \mathbf{h}_q$ and $\mu_{\mathbf{v}_q}$ is any probability measure on $\mathbf{v}_q$. Moreover, we assume that $\mu_{\mathbf{v}_q}$ is sufficiently regular as a function of $q$, and $2$-decreasing, namely that any linear function on $\mathbf{v}_q$ is $L^2(\mathbf{v}_q,\mu_{\mathbf{v}_q})$ (see Definition~\ref{d:decreasing}). Any such a measure is called \emph{adapted}.

Repeating this construction at each point we recover a differential operator, as in Eq.~\eqref{eq-microscopic}, that we call $L_{\mu_\V}$. It turns out that the latter does not depend on the choice of adapted measure, but only on the complement $\V$. In Section~\ref{s-luca-formula} we prove the following main result (which does not need any equiregularity assumption).
\begin{theorem}\label{t-lmu}
Let $TM = \distr \oplus \V$ and let $\mu_\V$ any adapted measure. Then $L_{\mu_\V}$ depends only on $\V$. Moreover, let $X_1,\ldots,X_k$ a local orthonormal frame for $\distr$, and $X_{k+1},\ldots,X_n$ a local frame for $\V$. Then
\begin{equation}
L_{\mu_\V} = \sum_{i=1}^k X_i^2 + \sum_{i,j=1}^k c_{ji}^j X_i,
\end{equation}
where the structural functions $c_{ij}^\ell \in C^\infty(M)$ are defined by $[X_i,X_j] = \sum_{\ell=1}^n c_{ij}^\ell X_\ell$ for all  $i,j=1,\ldots,n$. Finally, the convergence of Eq.~\eqref{eq-microscopic} is uniform on compact sets.
\end{theorem}
Thanks to Theorem~\ref{t-lmu}, in the following we use the notation $L^\V:=L_{\mu_\V}$.
\begin{rmk}
One of the byproducts of Theorem \ref{t-lmu} is the following. Fixing $\V$ is equivalent to assign a subspace of ``horizontal covectors'' $\mathbf{h}_q$ in $T^*_qM$. The expression of $L^\V$ at $q \in M$ is the same averaging only on the horizontal geodesics (i.e. with a measure $\mu_{\V_q} = \mu_{\mathbb{S}^{k-1}_q}\times \delta_{\mathbf{v}_q}$, where $\delta$ is the Dirac delta) or averaging on all possible geodesics with a measure of the type~\eqref{eq:product}. The particular choice $\mu_{\V_q} = \delta_{\mathbf{v}_q}$ recovers the construction of \cite{grong1,grong2,GordLae,GordLaeOlder}), where the authors choose a Riemannian extension and use this to define the space of horizontal covectors.
\end{rmk}
\begin{rmk}
We can rewrite the operator $L^\V$ in a more elegant way by introducing the concept of \emph{horizontal divergence}. As we will make precise in Section~\ref{horizontal-div}, the horizontal divergence of a vector field $X$ computes the infinitesimal change of volume, under the flow of a vector field $X$, of the standard parallelotope of $\distr$. To do this, we need a well defined projection $\pi: TM \to \distr$ that, indeed, requires the choice of a complement $\V$. We denote with $\div^\V(X)$ the horizontal divergence of $X$, and we have
\begin{equation}
L^\V = \div^\V \circ \grad.
\end{equation}
\end{rmk}

\subsection{The equivalence problem}
Once a volume $\omega$ on $M$ is chosen  (hence a macroscopic Laplacian is defined) and a complement $\V$ is fixed (hence a microscopic Laplacian is defined), it is natural to ask:  \\[2mm]
{\bf Q1:} Under which conditions on $\omega$ and $\V$ do we have $\Delta_\omega=L^\V$?\\[2mm]
In other words, we would like to know when a macroscopic Laplacian admits a microscopic interpretation and when a microscopic Laplacian can be written in divergence form (and hence is symmetric) w.r.t. some volume on the manifold. Moreover:\\[2mm]
{\bf Q2:} Given a volume $\omega$, is it possible to find a complement $\V$ such that $\Delta_\omega=L^\V$? If so, is it unique?\\[2mm]
This question is interesting since on any sub-Riemannian structure there is a smooth, intrinsic volume, Popp's one. Then an answer to {\bf Q2} gives a way to assign an intrinsic complement $\V$. A counting argument suggests that such a question has an affirmative answer. In fact, a volume form is given by a non-zero function, while a complement is given by $(n-k)k$ functions. However the answer is more complicated because  some integrability conditions must be taken in account.
A more specific question is the following:\\[2mm]
{\bf Q3:} Let $\V$ a complement, and let $\metr_\V$ a smooth scalar product on $\V$. Then the orthogonal direct sum $\metr \oplus \metr_\V$ is a \emph{Riemannian extension} of the sub-Riemannian structure (also called taming metric). Let $\omega_{\V}$ be the corresponding Riemannian volume. Is it true that for the sub-Riemannian diffusion operator $\Delta_{\omega_{\V}}=L^{ \V}$?\\[2mm]
This last question is even more interesting when it is possible to find an intrinsic Riemannian extension, i.e. some choice of $\V$ and $\metr_\V$ that depends only on the sub-Riemannian structure $(M,\distr,\metr)$. Intrinsic Riemannian extensions can be made in several cases (see for instance~\cite{diniz,Hladky-connection,Hladky-complement}). However, in general they are not known and (even if this is a non-apophantic statement) it is believed that they do not exist.

In Section~\ref{s:equiv} we answer to \textbf{Q1}, with the following theorem.
\begin{theorem}\label{t-general}
For any complement $\V$ and volume $\omega$, the macroscopic operator $\Delta_\omega$ and the microscopic operator $L^\V$ have the same principal symbol and no constant term. Moreover $L^\V = \Delta_\omega$ if and only if
\begin{equation}\label{eq:espressione}
\chi^{(\V,\omega)} := L^\V - \Delta_\omega = \sum_{i=1}^k \sum_{j=k+1}^n c_{ji}^j X_i +\grad( \theta)= 0,
\end{equation}
where $\theta = \log |\omega(X_1,\ldots,X_n)|$ and $c_{ij}^\ell$ are the structural functions associated with an orthonormal frame $X_1,\ldots,X_k$ for $\distr$ and a frame $X_{k+1},\ldots,X_n$ for $\V$.
\end{theorem}
The condition $\chi^{(\V,\omega)} = 0$ is a coordinate-free version of \cite[Th. 5.13]{GordLae}, which appeared online while this paper was under redaction. In \cite{grong1} the same condition is obtained when $\V$ is the orthogonal complement w.r.t. some Riemannian extension and $\omega$ is the corresponding Riemannian volume. The particularly simple form of Eq.~\ref{eq:espressione} in our frame formalism permits us to go further in the study of its solutions.

We address {\bf Q2} and {\bf Q3} for Carnot groups, contact and quasi-contact structures and, more generally, corank $1$ structures. Here we collect only the main results. For precise definition see the corresponding sections.

\subsubsection{Contact structures}

\begin{theorem}
Let $(M,\distr,\metr)$ be a contact sub-Riemannian structure. For any volume $\omega$ there exists a unique complement $\V$ such that $L^\V = \Delta_\omega$. In this case $\V = \spn\{X_0\}$, with
\begin{equation}
X_0 = \Z - J^{-1}\grad(\theta), \qquad \theta = \log| \omega(X_1,\ldots,X_k,Z)|,
\end{equation}
where $\Z$ is the Reeb vector field and $J:\distr \to \distr$ is the contact endomorphism.
\end{theorem}
Contact structures have a natural Riemannian extension, obtained by declaring the Reeb vector field a unit vector orthonormal to $\distr$. It turns out that the Riemannian volume of this extension is Popp's volume.
\begin{cor}
Let $\popp$ be the Popp's volume. The unique complement $\V$ such that $L^{\V} = \Delta_{\popp}$ is generated by the Reeb vector field. Moreover, $\popp$ is the unique volume (up to constant rescaling) with this property.
\end{cor}
\begin{rmk}
In the results above we always use the normalization $\|J\| = 1$ (this fixes the contact form up to a sign). We stress that if we choose a different normalization the Reeb field would be different.
\end{rmk}
In Section~\ref{s:integrability} we also discuss the inverse problem, namely for a fixed $\V$, find a volume $\omega$ such that $L^\V = \Delta_\omega$. This is a more complicated problem (and in general has no solution). In the contact case, thanks to the non-degeneracy of $J$, we find explicitly a necessary and sufficient condition.
\begin{prop}
Let $\V = \spn\{X_0\}$. Define the one-form $\alpha := \frac{i_{X_0} d\eta}{\eta(X_0)}$ and the function $g=\frac{d\alpha \wedge\eta \wedge (d\eta)^{d-1}}{\eta\wedge (d\eta)^d }$. Then there exists a volume $\omega$ such that $L^{\V} = \Delta_\omega$ if and only if
\begin{equation}
d\alpha - dg \wedge \eta - g d\eta = 0.
\end{equation}
In this case, $\omega$ is unique up to constant rescaling.
\end{prop}

\subsubsection{Carnot groups}

By definition, Carnot groups are particular cases of left-invariant sub-Riemannian structures (see Definition~\ref{d:left-invariant}). Then it is natural to choose a left-invariant volume i.e. a volume proportional to the Haar's one (for example, Popp's volume $\popp$) and left-invariant complements $\V$. Contrary to the contact case (where we always have existence and uniqueness of $\V$ for any fixed $\omega$) here we lose uniqueness.
\begin{prop}
For any Carnot group $G$, we have $\Delta_{\popp} = L^{\V_0}$, where $\V_0$ is the left invariant complement
\begin{equation}
\V_0 := \mathfrak{g}_2 \oplus \dots\oplus \mathfrak{g}_m,
\end{equation}
where $\mathfrak{g}_i$ denotes the $i$-th layer of the Carnot group.
\end{prop}
Any left-invariant complement is the graph of a linear map $\ell: \V_0 \to \distr$, that is $\V_\ell := \{X+ \ell( X)\mid X \in \V_0\}$.
\begin{prop}
$\Delta_\popp = L^\V$ if and only if $\V = \V_\ell$, with
\begin{equation}
\tr(\ell \circ \ad_{X_i}) = 0, \qquad \forall i=1,\ldots,k.
\end{equation}
\end{prop}
The equation above has non-unique solutions $\ell$ in many cases, as we show in Section~\ref{s:carnot}.

\subsubsection{Quasi-contact structures}

For all the relevant definitions, we refer to Section~\ref{s:quasicontact}. We only stress that, analogously to the contact case, one can define a quasi-contact endomorphism $J:\distr \to \distr$ that, in the quasi-contact case, is degenerate (and its kernel is usually assumed to have dimension $1$). In this case, if for some $\omega$ there is $\V$ such that $L^\V = \Delta_\omega$, then $\V$ is \emph{never} unique. In particular, we have the following.
\begin{theorem}
For any fixed $\omega$, the space of $\V$ such that $L^\V = \Delta_\omega$ is an affine space over $\ker J$.
\end{theorem}
Surprisingly, we not only lose uniqueness but also existence. Consider the quasi-contact structure on $M= \R^4$, with coordinates $(x,y,z,w)$ defined by $\distr = \ker \eta$ with
\begin{equation}
\eta = \frac{g}{\sqrt{2}}\left(dw-\frac{y}{2}dx+\frac{x}{2}dy\right), \qquad \text{with} \qquad g = e^z.
\end{equation}
($g$ can be any strictly monotone, positive function). The metric is defined by the global orthonormal frame:
\begin{equation}
X = \frac{1}{\sqrt{g}}\left( \partial_x + \frac{1}{2}y \partial_w \right), \qquad Y = \frac{1}{\sqrt{g}}\left(\partial_y - \frac{1}{2} x \partial_w\right), \qquad \Z = \frac{1}{\sqrt{g}} \partial_z .
\end{equation}
Choose $\omega=\popp$, the Popp's volume, that is 
\begin{equation}
\popp = \tfrac{g^{5/2}}{\sqrt{2}}dx\wedge dy \wedge dz \wedge dw.
\end{equation}
\begin{prop}
For the structure above $L^\V \neq \Delta_\popp$ for any choice of the complement $\V$.
\end{prop}
Even though this is out of the scope of the present paper, it turns out that this is the typical (generic) picture in the quasi-contact case. This result is quite surprising because, on sub-Riemannian structures with dimension smaller or equal then $4$, there is a unique N-intrinsic volume up to scaling, and a unique N-intrinsic Laplacian given by $\Delta_\popp$. In our example, this unique N-intrinsic Laplacian has \emph{no} compatible complement or, in other words, the macroscopic diffusion operator has no microscopic counterpart. 

On a quasi-contact manifold one can build an analogue of the Reeb vector field (actually, a one-parameter family parametrized by the distinct eigenvalues of $J$), that provides a standard Riemannian extension (see Section~\ref{s:quasireeb}). It turns out that the Riemannian volume of these extensions is once again Popp's volume. Thus, the above non-existence result provides also an answer to {\bf Q3}: the macroscopic operator $\Delta_\omega$ provided by the quasi-Reeb Riemannian extension is not the microscopic operator $L^\V$ provided by the quasi-Reeb complement.

\subsubsection{Convergence of random walks}

Finally, in Section \ref{a:randomwalk}, we provide the probabilistic side of the construction of the microscopic Laplacian. Namely, as we see in Eq.\ \eqref{eq-microscopic}, the microscopic Laplacian is built from the scaling limit of single step of a random walk. In Theorem~\ref{AppendixConvergence2}, we show that this convergence of a single-step (in the sense of the convergence of the induced operator on smooth functions) can be promoted to the convergence of the random walks to the diffusion generated by the limit operator. This connects the microscopic Laplacian back to the heuristic discussion in the Section \ref{RiemIntro}. Moreover, in Theorem~\ref{AppendixConvergence}, the convergence of a sequence of random walks to the diffusion generated by some appropriate second-order operator is established in much more generality than required for the present paper (for example, the probability measure for choosing a co-vector at each step of the walk can be supported on the entire co-tangent space), with an eye toward a wider variety of possible approximation schemes.

\subsection{Comparison with recent literature}\label{s:literature}

In the Riemannian setting, different authors have analyzed the convergence of geodesic random walks to the Brownian motion generated by the Laplace-Beltrami operator. In particular, in \cite{PinksyRiem}, Pinsky considers a process (a random flight in our terminology, that he calls the isotropic transport process) that, starting from $x \in M$, moves along a geodesic with initial vector uniformly chosen on the unit-sphere, for a randomly chosen, exponentially-distributed time, after which a new vector is chosen. The lift of this walk to the tangent bundle is a Markov process that can be understood in terms of the corresponding semi-group. Using this, the author shows that under appropriate (parabolic) rescaling, the semi-group of the random flight converges to the Brownian motion semigroup (the generator of which is the Laplace-Beltrami operator). In \cite{LebeauRiem}, Lebeau and Michel investigate a random walk on smooth, compact, connected Riemannian manifolds, that at each step jumps to a uniformly chosen (according to the Riemannian volume measure) point in the ball of radius $h$ around the current position. A natural modification of such a random walk (based on the Metropolis algorithm) approximates Brownian motion on the manifold when $h$ is sent to zero and time is rescaled appropriately. Moreover, the authors consider the transition operator of the random walk, and prove that its rescaled spectrum approximates the spectrum of the Laplace-Beltrami operator. This provides a sharp rate of convergence of this random walk to its stationary distribution.

Horizontal Laplacians and diffusions on sub-Riemannian manifolds and random walk approximations to these have appeared several times in the literature recently. We make use of our notation and terminology in describing these results, in order to make the connection with the present paper clearer.

In \cite{LebeauSriem}, Lebeau and Michel study the spectral theory of a reversible Markov chain associated to a random walk on a sub-Riemannian manifold $M$, where, at each step, one walks along the integral lines (not geodesics, in general) of a fixed set of divergence-free vector fields $X_1,\ldots,X_k$ (with respect to some fixed volume $\omega$). This random walk depends on a parameter $h$. In particular, they prove the convergence when $h\to 0$ to an associated hypoelliptic diffusion, which, being that the vector fields are divergence free, coincides with $\Delta_\omega$. In a similar spirit as above, they also consider the rate of convergence to equilibrium of a random walk of this type.

In \cite{GordLaeOlder}, Gordina and Laetsch use a Riemannian extension of a sub-Riemannian metric to determine an orthogonal complement $\V_{\g}$ to $\distr$, which is equivalent to determining a subspace of horizontal co-vectors. This allows them to define a horizontal Laplacian by averaging over second-derivatives in the horizontal directions. The result is, in our terminology, a microscopic Laplacian that, in fact, depends only on the choice of complement $\V_{\g}$. They then introduce a corresponding family of horizontal random flights (in our terminology, though they call them random walks-- see the discussion in Appendix~\ref{a:randomwalk}), given by choosing a horizontal co-vector of fixed length uniformly at random, and then following the resulting geodesic at a constant speed for an exponentially distributed length of time, before repeating the procedure. They show that, under the natural parabolic scaling of the length of the co-vectors and the mean of the exponential ``travel times,'' these random flights converge to the diffusion generated by the horizontal Laplacian. To do this, they use resolvent formalism in order to prove the convergence of the relevant semi-groups, in a similar vein to the paper of Pinsky \cite{PinksyRiem} mentioned above.

In \cite{GordLae}, Gordina and Laetsch give a more systematic discussion of horizontal Laplacians relative to a choice of vertical complement. In particular, they take up the question of when a horizontal Laplacian relative to a choice of complement is equal to the divergence of the horizontal gradient with respect to some volume (or, as we have framed the question here, when the macroscopic Laplacian $\Delta_\omega$ is equal to the microscopic Laplacian $L^\V$). They give (see \cite[Theorem 5.13]{GordLae}) a condition for this that is equivalent to our Theorem \ref{t:compatibility}, though stated in terms of some Riemannian extension of the sub-Riemannian metric. Having given this result, they consider the application to three concrete examples of 3-dimensional Lie groups, the Heisenberg group, $SU(2)$, and the affine group.

In \cite{grong1}, Grong and Thalmaier define the horizontal Laplacian, corresponding to what we have called the microscopic Laplacian, by extending the sub-Riemannian metric to a Riemannian metric, considering an associated connection, and then taking the trace of a projection of the Hessian (see  \cite[Def. 2.1]{grong1}). The resulting operator depends only on the orthogonal complement $\V_{\g}$ to $\distr$ with respect to the Riemannian extension $\g$. Random walks are not considered in this approach, although it also gives a construction of the horizontal Laplacian associated to a choice of horizontal subspace. They do consider the question of when their horizontal Laplacian (with respect to $\V_{\g}$) is equal to the Laplacian determined by the Riemannian volume $\mathcal{R}_\g$ of the extension metric, and the result is essentially our Theorem \ref{t:compatibility} applied to this situation, presented in terms of the extension metric $\g$ (and with a similar proof). Nonetheless, the primary interest of \cite{grong1} is curvature-dimension inequalities for sub-Riemannian structures associated to certain Riemannian foliations, in the spirit of Baudoin and Garofalo \cite{FabriceJEMS}, not in a discussion of sub-Laplacians per se. Indeed, they go on to select, among all possible complements (and then associated diffusions) special ones, that are integrable and ``metric-preserving'' in a way suitable for their purpose.

In light of the above, the novelty of the present paper lies primarily in the following areas. We abandon the Riemannian formalism, and we use frame formalism, which seems better suited to the sub-Riemannian context (for instance, we get operators that depend on the choice of complement $\V$ by construction). Also, having given the condition for the equivalence of macroscopic and microscopic Laplacians (``the equivalence problem'') in frame formalism (again, Theorem \ref{t:compatibility}), we proceed to analyze many broad classes of examples in some detail (for example, any Carnot group and any corank-$1$ structure). In the contact case, we solve completely the equivalence problem. We do the same in the quasi-contact case, for the Popp volume $\omega = \popp$. This leads to the perhaps-surprising fact that there are 4-dimensional quasi-contact structures for which the macroscopic Laplacian with respect to the Popp volume, which is in a precise sense the canonical volume, cannot be realized as a microscopic Laplacian (indeed, this is generically the case, though the complete proof is beyond the scope of the present paper). More generally, we discuss to what extent volume measures on a sub-Riemannian manifold can be thought of as canonical, which has implications for the degree to which a macroscopic Laplacian can be thought of as canonical.

On the probabilistic side, we allow random walks where the choice of co-vector is supported on the entire unit cylinder in the co-tangent space, rather than just on the horizontal subspace (relative to some choice of complement). (This really just says that sub-Riemannian geometry is rather insensitive to adding some independent vertical component to the initial co-vectors in a random walk, which is not surprising.) Also, our Theorem \ref{AppendixConvergence} gives a general convergence result for random walks on (sub)-Riemannian manifolds, going beyond the particular type of random walks considered elsewhere in the paper. (This theorem has been recently used in \cite{OurVolume} to prove convergence of a class of random walks which, in a suitable way, sample the ambient volume.) From a technical perspective, this convergence of random walks is proved using martingale methods, rather than the semi-group approach mentioned above.


\section{Preliminaries}\label{s:prel}

We discuss some preliminaries in sub-Riemannian geometry. We essentially follow \cite{nostrolibro}, but see also \cite{montgomerybook,riffordbook,Jea-2014}. 
\begin{definition}
A \emph{sub-Riemannian manifold} is a triple $(M,\distr,\metr)$ where:
\begin{itemize}
\item $M$ is smooth, connected manifold;
\item $\distr \subset TM$ is a smooth distribution of constant rank $k < n$, satisfying \emph{H\"ormander’s condition}.
\item $\metr$ is a smooth scalar product on $\distr$: for all $q \in M$, $\metr_q$ is a positive definite quadratic form on $\distr_q$, smooth as a function of $q$.
\end{itemize}
This definition does not include Riemannian structures, for which $k =n$ and $\distr = TM$. We use the term \emph{(sub)-Riemannian} to refer to structures $(M,\distr,\metr)$ that are either Riemannian or sub-Riemannian.
\end{definition}
\begin{definition}
Define $\distr^{1}:=\distr$, $\distr^{i+1}:=\distr^{i}+[\distr^{i},\distr]$, for every $i\geq1$. A sub-Riemannian manifold is said to be \emph{equiregular} if for each $i\geq1$, the dimension of $\distr^{i}_{q}$ does not depend on the point $q\in M$. 
\end{definition}
Smooth sections of $\distr$ are called \emph{horizontal vector fields}. Hormander's condition guarantees that any two points in $M$ can be joined by a Lipschitz continuous curve whose velocity is a.e. in $\distr$ (Chow-Rashevskii theorem). We call such curves \emph{horizontal}. Horizontal curves $\gamma : I \to M$ have a well-defined length, given by
\begin{equation}
\ell(\gamma) = \int_I \|\gamma(t)\| dt,
\end{equation}
where $\| \cdot \|$ is the norm induced by the (sub)-Riemannian scalar product. The \emph{Carnot-Caratheodory distance} between two points $p,q \in M$ is
\begin{equation}
d(p,q) = \inf\{ \ell(\gamma) \mid \gamma \text{ horizontal curve connecting $q$ with $p$} \}.
\end{equation}
This distance turns $(M,\distr,\metr)$ into a metric space that has the same topology of $M$.
\begin{definition}
The \emph{Hamiltonian function} $H: T^*M \to \R$ associated with a sub-Riemannian structure is
\begin{equation}
H(\lambda) := \frac{1}{2} \sum_{i=1}^k \langle\lambda,X_i\rangle^2,
\end{equation}
for any choice of a local orthonormal frame $X_1,\ldots,X_k$ of horizontal fields, i.e. $\metr(X_i,X_j) =\delta_{ij}$.
\end{definition}
On each fiber, $H$ is non-negative quadratic form, and provides a way to measure the ``length'' of covectors. By setting $H_q:= H|_{T_q^*M}$, one can check that $\ker H_q = \distr_q^\perp$, the set of covectors that vanish on the distribution. In the Riemannian case, $H$ is precisely the fiber-wise inverse of the metric $\metr$.

Let $\sigma$ be the natural symplectic structure on $T^*M$, and $\pi$ the projection $\pi: T^*M \to M$. The \emph{Hamiltonian vector field} $\vec{H}$ is the unique vector field on $T^*M$ such that $dH = \sigma(\cdot,\vec{H})$. Integral lines of $\vec{H}$ are indeed smooth curves on the cotangent bundle that satisfy Hamilton's equations $\dot{\lambda}(t) = \vec{H}(\lambda(t))$.
\begin{definition}
The projections $\gamma(t) = \pi(\lambda(t))$ of integral lines of $\vec{H}$ are called \emph{normal (sub)-Riemannian geodesics}.
\end{definition}
Normal geodesics are indeed smooth and, as in the Riemannian case, are locally minimizing (i.e. any sufficiently small segment of $\gamma(t)$ minimizes the distance between its endpoints).
For any $\lambda \in T^*M$ we consider the associated normal geodesic $\gamma_\lambda(t)$, obtained as the projection of the integral line $\lambda(t)$ of $\vec{H}$ with initial condition $\lambda(0) = \lambda$. The \emph{initial covector} $\lambda$ plays the same role in sub-Riemannian geometry of the initial vector of Riemannian geodesics, with the important difference that an infinite number of distinct sub-Riemannian normal geodesics, with different initial covector, have the same initial vector. Notice that the Hamiltonian function, which is constant on integral lines $\lambda(t)$, measures the speed of the normal geodesic: 
\begin{equation}
2H(\lambda) = \|\dot\gamma_\lambda(t)\|^2,\qquad \lambda \in T^*M.
\end{equation}
We are ready for the following important definition.
\begin{definition}
Let $D_q\subseteq [0,\infty) \times T_q^*M$ the set of the pairs $(t,\lambda)$ such that the normal geodesic with initial covector $\lambda$ is well defined up to time $t$. The (sub)-Riemannian \emph{exponential map} (at $q \in M$) is the map  $\exp_q: D_q \to M$ that associates with $(t,\lambda)$ the point $\gamma_\lambda(t)$.
\end{definition}
When clear, we suppress the initial point $q$ in $\exp_q$. It is easy to show that, for any $\alpha >0 $, we have
\begin{equation}
\gamma_{\alpha \lambda}(t) = \gamma_{\lambda}(\alpha t).
\end{equation}
This rescaling property, due to the fact that $H$ is fiber-wise homogeneous of degree $2$, justifies the restriction to the subset of initial covectors lying in the level set $2H=1$. 
\begin{definition}
The \emph{unit cotangent bundle} is the set of initial covectors such that the associated normal geodesic has unit speed, namely
\begin{equation}
\cyl := \{\lambda \in T^*M \mid 2H(\lambda) = 1\} \subset T^*M.
\end{equation}
\end{definition}
\begin{rmk}
We stress that, in the sub-Riemannian case, the non-negative quadratic form $H_q = H|_{T_q^*M}$ has non-trivial kernel. It follows that the fibers $\cyl_q$ are cylinders and thus, non-compact, in sharp contrast with the Riemannian case (where the fibers $\cyl_q$ are spheres). 
\end{rmk}

For any $\lambda \in \cyl$, the corresponding geodesic $\gamma_\lambda(t)$ is parametrized by arc-length and $\ell(\gamma|_{[0,T]}) = T$. Even if $\cyl$ is not compact, all arc-length parametrized geodesics are well defined for a sufficiently small time. The next Lemma is a consequence of the form of Hamilton's equations and the compactness of small balls.
\begin{lemma}\label{l:prolong}
There exists $\eps >0$ such that $[0,\eps) \times \cyl_q \subseteq D_q$. In other words all arc-length parametrized normal geodesics $\gamma_\lambda(t)$ are well defined on the interval $[0,\eps)$.
\end{lemma}

In the Riemannian case, the gradient of a function is a vector field that points in the direction of the greatest rate of increase of the function. The generalization to the (sub)-Riemannian setting is straightforward.
\begin{definition}\label{d:grad}
Let $f \in C^\infty(M)$. The \emph{horizontal gradient} $\grad(f) \in \Gamma(\distr)$ is defined by
\begin{equation}
df(X) = \metr(\grad(f),X), \qquad \forall X \in \Gamma(\distr).
\end{equation}
\end{definition}
Since, in the Riemannian case, it is the usual gradient, this notation will cause no confusion.

\subsection{Computations with frames}

If $E$ is a smooth vector bundle over $M$, the symbol $\Gamma(E)$ denotes the $C^\infty(M)$-module of smooth sections of $E$. Horizontal vector fields are then elements of $\Gamma(\distr)$. 

In sub-Riemannian geometry, computations are most effectively done in terms of orthonormal frames (Riemannian normal coordinates are not available in general). Then, let $X_1,\ldots,X_k$ a (local) orthonormal frame for the sub-Riemannian structure. Moreover, consider some complement $X_{k+1},\ldots,X_n$, namely a local frame that completes $X_1,\ldots,X_k$ to a local frame for $TM$. Let $c_{ij}^\ell \in C^\infty(M)$ (the \emph{structural functions}) be defined by:
\begin{equation}
[X_i,X_j] = \sum_{\ell=1}^n c_{ij}^\ell X_\ell, \qquad i,j = 1,\ldots,n.
\end{equation}
Now define the functions $h_i : T^*M \to \R$ (that are linear on fibers) as:
\begin{equation}
h_i(\lambda) := \langle \lambda, X_i\rangle, \qquad i = 1,\ldots,n.
\end{equation}
We have
\begin{equation}
H = \frac{1}{2}\sum_{i=1}^k h_i^2, \qquad \vec{H} = \sum_{i=1}^k h_i \vec{h}_i,
\end{equation}
where $\vec{h}_i$ is the Hamiltonian vector field associated with $h_i$, namely $\sigma(\cdot,\vec{h}_i) = dh_i$. Indeed, for any fixed $q$, the restriction of $(h_1,\ldots,h_n) : T_{q}^*M \to \R^n$ gives coordinates to $T_{q}^*M$, associated with the choice of $X_1,\ldots,X_n$. In terms of these coordinates the fibers of the unit cotangent bundle are
\begin{equation}
\cyl_q := T_{q}^*M \cap \cyl = \{(h_1,\ldots,h_k,h_{k+1},\ldots,h_n) \mid h_1^2+\ldots+ h_k^2 = 1\}  \simeq \mathbb{S}^{k-1} \times \R^{n-k}.
\end{equation}
The last identification depends on the choice of the frame $X_1,\ldots,X_n$. Finally, for any normal geodesic $\gamma_\lambda(t)$
\begin{equation}
\dot{\gamma}_\lambda(t) = \pi_* \dot{\lambda}(t) = \pi_* \vec{H}(\lambda(t)) = \sum_{i=1}^k h_i(\lambda(t)) X_i,
\end{equation}
where we used the fact that $\pi_* \vec{h}_i = X_i$.

\subsection{A Taylor expansion with frames}

In this section we prove a Taylor expansion for smooth functions along normal geodesics. This formula will play a key role in the following. To keep a compact notation we often employ the \emph{generalized Einstein's convention}: repeated indices are summed, albeit on different ranges: Greek indices ($\alpha,\beta,\gamma,\dots$) from $1$ to $n$; Latin indices ($i,j,\ell,\dots$) from $1$ to $k$; barred Latin indices ($\bar{\imath},\bar{\jmath},\bar{\ell},\dots$) from $k+1$ to $n$.

\begin{lemma}\label{l:taylor}
Let $\phi \in C^\infty(M)$ and consider the geodesic $\gamma_\lambda(t)$, emanating from $q$. Then
\begin{equation}
\phi(\gamma_\lambda(t)) = \phi(q)+ t h_i X_i(\phi)(q) + \frac{1}{2}t^2 \left[h_j  c_{ji}^\alpha h_\alpha X_i(\phi)(q) + h_i h_j X_j(X_i(\phi))(q)\right] + t^3 r_\lambda(t),
\end{equation}
where the initial covector $\lambda \in T_q^*M$ has coordinates $(h_1,\ldots,h_n)$ and $r_\lambda(t)$ is a remainder term. Moreover for any $q \in M$, there exists an $\eps_0 >0$ and constants $A,B_{\bar\ell},C_{\bar\imath\bar\jmath} \geq 0$ such that, for any $\lambda \in \cyl$ with $\pi(\lambda) \in B(q,\eps_0)$ (the metric ball with center $q$ and radius $\eps_0$) we have
\begin{equation}
|r_\lambda(t)|\leq A + B_{\bar{\ell}} |h_{\bar\ell}| + C_{\bar\imath\bar\jmath} |h_{\bar\imath}| |h_{\bar\jmath}|, \qquad \forall t \leq \eps_0.
\end{equation}
\end{lemma}
\begin{rmk}
In coordinates $\lambda = (h_1,\ldots,h_n)$ with $h_1^2+\ldots+h_k^2 = 1$. Thus, the estimate above shows how the remainder term depends on the ``unbounded'' coordinates $h_{k+1},\ldots,h_n$ of the initial covector.
\end{rmk}
\begin{proof}
The geodesic $\gamma_\lambda(t)$ is the projection of the integral curve $\lambda(t)$. Its initial covector $\lambda \in T_{q}^*M$ has coordinates $(h_1,\ldots,h_n) \in \R^n$. We drop the subscript $\lambda$, since it's fixed. As we described above,
\begin{equation}\label{eq:firstder}
\frac{d}{d t}\phi(\gamma(t)) = \dot\gamma(t) ( \phi) = h_i(t) X_i(\phi)(\gamma(t)),
\end{equation}
where $h_i(t)$ is a shorthand for $h_i(\lambda(t))$. Similarly, for any function $g \in C^\infty(T^*M)$, we set $g(t)= g(\lambda(t))$. With this notation, Hamilton's equations are:
\begin{equation}
\dot{g}(t) = \{H,g\}(t),
\end{equation}
where $\{\cdot,\cdot\}$ denotes the Poisson bracket, and the dot the derivative w.r.t. $t$. For $h_i:T^*M \to \R$ we get
\begin{equation}	
\dot h_i = \{H,h_i\} = h_j\{h_j,h_i\} = h_j c_{ji}^\alpha h_\alpha,
\end{equation}
where we suppressed the explicit evaluation at $t$. We apply the chain rule to Eq.~\eqref{eq:firstder} and we get
\begin{equation}\label{eq:seconder}
\frac{d^2}{d t^2}\phi(\gamma(t)) = \dot{h}_i X_i(\phi) + h_i h_j X_j(X_i(\phi)) = h_j c_{ji}^\alpha h_\alpha X_i(\phi) + h_i h_j X_j(X_i(\phi)).
\end{equation}
Evaluating at $t=0$ we get the second order term. Now let $t \leq \eps_0$. The remainder of Taylor's expansion in Lagrange's form is:
\begin{equation}
r_\lambda(t) = \frac{1}{3!}\left.\frac{d^3}{d t^3}\right|_{t=t_*}\phi(\gamma(t)), \qquad t_* \in [0,t].
\end{equation}
To compute it we apply the chain rule to Eq.~\eqref{eq:seconder}. By Hamilton's equations $\dot{h}_\alpha = h_i c_{i\alpha}^\beta h_\beta$, we reduce it to a polynomial in $h_1(t),\ldots,h_n(t)$, structural functions $c_{\alpha\beta}^\gamma$ and their first derivatives $X_i(c_{\alpha\beta}^\gamma)$:
\begin{equation}
\begin{aligned}
\frac{d^3}{d t^3}\phi(\gamma(t))& = h_\ell c_{\ell j}^\beta h_\beta c_{ji}^\alpha h_\alpha X_i(\phi) + h_j h_\ell X_\ell(c_{ji}^\alpha) h_\alpha X_i(\phi) + h_j c_{ji}^\alpha h_\ell c_{\ell \alpha}^\beta h_\beta X_i(\phi) \\
& \quad + h_j c_{ji}^\alpha h_\alpha h_\ell X_\ell( X_i(\phi))+ h_\ell c_{\ell i}^\alpha h_\alpha h_j X_j(X_i(\phi)) \\
& \quad + h_i h_\ell c_{\ell j}^\alpha h_\alpha X_j(X_i(\phi)) + h_i h_j h_\ell X_\ell(X_j(X_i(\phi))).
\end{aligned}
\end{equation}
We stress that everything on the r.h.s. is computed at $t$. Since $\lambda \in \cyl_q$, and $2H$ is a constant of the motion, each $|h_i(t)| \leq 1$ for $i=1,\ldots,k$. Fix $q \in M$, and consider the sub-Riemannian closed ball $\overline{B(q,2\eps_0)}$, that is compact for sufficiently small $\eps_0$. If $\lambda \in \cyl \cap \pi^{-1}(B(q,\eps_0))$, the length parametrized geodesic $\gamma_\lambda$ does not exit the compact $\overline{B(q,2\eps_0)}$. Therefore also the structural functions and their derivatives are bounded by their maximum on $\overline{B(q,2\eps_0)}$. The only a priori uncontrolled terms are then the $h_{\bar\imath}(t)$, for $\bar\imath=k+1,\ldots,n$. Thus:
\begin{equation}\label{eq:remainder_t}
|r_\lambda(t)| \leq P + Q_{\bar{\ell}} |h_{\bar{\ell}}(t)| + R_{\bar{\imath}\bar{\jmath}} |h_{\bar{\imath}}(t)| |h_{\bar{\jmath}}(t)|, \qquad \forall t \leq \eps_0,
\end{equation}
where $P,Q_{\bar{\ell}}, R_{\bar{\imath}\bar{\jmath}} \geq 0$ are constants. Gronwall's Lemma shows that, for $t\leq \eps_0$ we have
\begin{equation}
h_{\bar\imath}(t) \leq D_{\bar\imath} + E_{\bar\imath\bar\jmath} | h_{\bar{\jmath}}(0)|, \qquad \bar\imath = k+1,\ldots,n,
\end{equation}
for constants $D_{\bar\imath},E_{\bar\imath\bar\jmath} \geq 0$. By plugging this last equation into Eq.~\eqref{eq:remainder_t}, we obtain the result.
\end{proof}


\section{On volume forms in (sub)-Riemannian geometry}\label{s-intrinsic}

The aim of this section is to introduce a rigorous definition of N-intrinsic volume as the one that ``depends only on the first order approximation of the (sub)-Riemannian structure'' (see the discussion in Section~\ref{s:intro}).

Let $M$ be an orientable $n$-dimensional smooth manifold. A (smooth) \emph{volume form} is a positive $n$-form $\omega$, that is $\omega(X_1,\ldots,X_n)  > 0$ for any positively-oriented local frame $X_1,\ldots,X_n$. Any volume form $\omega$ defines a positive measure on Borel sets of $M$, that we still call $\omega$.   Let $f \in C^\infty(M)$, and $U$ a Borel set. In this way, $\int_U f \omega$ denotes the integral over $U$ of the function $f$ with respect to the measure induced by $\omega$. 

\begin{definition}\label{d:diver}
Let $X \in \Gamma(TM)$ and $\omega$ be a volume form. The \emph{divergence} $\div_\omega (X)$ is the function defined by
\begin{equation}
\mathcal{L}_X \omega = \div_\omega(X) \omega,
\end{equation}
where $\mathcal{L}_X$ denotes the Lie derivative in the direction of $X$.
\end{definition}

The next two lemmas are easy consequences of the definition.
\begin{lemma}
Let $f \in C_0^\infty(M)$ and $\omega$ be a volume form. Then
\begin{equation}
\int_M f \div_\omega(X) \omega = - \int_M df(X) \omega.
\end{equation}
\end{lemma}
\begin{lemma}
Let $f \in C^\infty(M)$, with $f \neq 0$ and $\omega$ a volume form. Let $\omega' = f\omega$. Then
\begin{equation}
\div_{\omega'}(X) = \div_{\omega}(X) + X(\log |f|).
\end{equation}
\end{lemma}

\subsection{The nilpotent approximation}

A key ingredient is the definition of \emph{nilpotent approximation}: a (sub)-Riemannian structure on the tangent space at a point $T_q M$ which, in a sense, is the ``first order approximation'' of the (sub)-Riemannian structure. Let $\distr$ be an equiregular, bracket-generating distribution. The \emph{step} of $\distr$ is the first $m$ such that $\distr_q^m = T_q M$. 

\begin{definition}
The \emph{nilpotentization} of $\distr$ at a point $q \in M$ is the graded vector space
\begin{equation}
\hat{M}_q := \distr_q \oplus \distr_q^2/\distr_q \oplus \dots \oplus \distr^m_q/\distr_q^{m-1}.
\end{equation}
\end{definition}
The vector space $\hat{M}_q$ can be endowed with a Lie algebra structure, which respects the grading, induced by the Lie bracket, as follows. Let $X_q \in \distr^i_q/\distr^{i-1}_q$ and $Y_q \in \distr^j_q/\distr^{j-1}_q$. Let $X,Y \in \Gamma(TM)$ smooth extensions of any representative of $X_q,Y_q$. Then the Lie product between $X_q$ and $Y_q$ is
\begin{equation}
[X_q,Y_q] := [X,Y]_q \mod \distr_q^{i+j-1} \in \distr_q^{i+j}/\distr_q^{i+j-1}.
\end{equation}
Under the equiregularity assumption, this product is well defined and does not depend on the choice of the representatives and of the extensions. Observe that the graded Lie algebra $\hat{M}_q$ is nilpotent. Then there is a unique connected, simply connected group, such that its Lie algebra is $\hat{M}_q$ (that we identify with the group itself). The global, left-invariant vector fields obtained by the group action on any orthonormal basis of $\distr_q \subset \hat{M}_q$ give $\hat{M}_q$ the structure of a left-invariant (sub)-Riemannian manifold, which is called the \emph{nilpotent approximation} of the sub-Riemannian structure $(M,\distr,\g)$ at the point $q$. We stress that:
\begin{itemize}
\item the base manifold is the vector space $\hat{M}_q \simeq T_q M$ (the latter identification is \emph{not} canonical);
\item the distribution and the scalar product are obtained by extending $\distr_q,\g_q$ using the left action of the group.
\end{itemize}

It can be proved that $\hat{M}_q$ is isometric (as a metric space) to the Gromov tangent cone at $q$ of the metric space $(M,d)$ where $d$ is the Carnot-Carath\'eodory distance induced by the (sub)-Riemannian structure (see~\cite{montgomerybook,bellaiche,mitchell}).

\subsection{Popp's volume}\label{s:Popp}

In this section we provide the definition of Popp's volume. Our presentation follows closely the one of \cite{nostropopp,montgomerybook}. The definition rests on the following lemmas.

\begin{lemma} \label{l:mont1} 
Let $E$ be an inner product space, and let $\pi:E\to V$ be a surjective linear map. Then $\pi$ induces an inner product on $V$ such that the length of $v \in V$ is
\begin{equation} \label{eq:final}
\|v\|_V = \min\{ \|e\|_E \text{ s.t. } \pi(e) = v \}.
\end{equation}
\end{lemma}

\begin{lemma} \label{l:mont2} 
Let $E$ be a vector space of dimension $n$ with a flag of linear subspaces $\{0\} = F^0 \subset F^1\subset F^2 \subset \ldots\subset F^m = E$. Let $\gr(F) = F^1\oplus F^2/F^1\oplus \ldots \oplus F^m/F^{m-1}$ be the associated graded vector space. Then there is a canonical isomorphism $\theta: \wedge^n E \to \wedge^n \gr(F)$. 
\end{lemma}
The proofs can be found in \cite{nostropopp}. We report here the proof of Lemma~\ref{l:mont2} since it contains the definition of the important map $\theta$.
\begin{proof}
We only give a sketch of the proof. For $0\leq i \leq m$, let $k_i:= \dim F^i$. Let $X_1,\dots,X_n$ be an adapted basis for $E$, i.e. $X_1,\dots, X_{k_i}$ is a basis for $F^i$. We define the linear map $\widehat{\theta}: E \to \gr(F)$ which, for $0\leq j \leq m-1$, takes $X_{k_j+1}, \dots, X_{k_{j+1}}$ to the corresponding equivalence class in $F^{j+1}/F^j$. This map is indeed a non-canonical isomorphism, which depends on the choice of the adapted basis. In turn, $\widehat{\theta}$ induces a map $\theta : \wedge^n E \to \wedge^n \gr(F)$, which sends $X_1\wedge\ldots\wedge X_n$ to $\widehat{\theta}(X_1)\wedge\ldots\wedge\widehat{\theta}(X_n)$. The proof that $\hat{\theta}$ does not depend on the choice of the adapted basis is a straightforward check, and boils down to the fact that two different adapted basis are related by an upper triangular matrix (see \cite{nostropopp,montgomerybook}).
\end{proof}

The idea behind Popp's volume is to define an inner product on each $\distr^i_q/\distr^{i-1}_q$ which, in turn, induces an inner product on the orthogonal direct sum $\hat{M}_q$. The latter has a natural volume form, such that its value on any oriented orthonormal basis is $1$. Then, we employ Lemma~\ref{l:mont2} to define an element of $(\wedge^n T_q M)^*$, which is Popp's volume form computed at $q$.

Fix $q \in M$. Then, let $v,w \in \distr_q$, and let $V,W$ be any horizontal extensions of $v,w$. Namely, $V,W \in \Gamma(\distr)$ and $V(q) =v$, $W(q) = w$. Let $2\leq i \leq m$. The linear maps $\pi_i: \otimes^i \distr_q \to \distr_q^i/\distr_q^{i-1}$
\begin{equation}\label{eq:pimap}
\pi_i(v_1\otimes\dots\otimes v_i) = [V_1,[V_2,\dots,[V_{i-1},V_i]]]_q \mod \distr^{i-1}_q,
\end{equation}
are well defined and do not depend on the choice of the horizontal extensions $V_1,\dots,V_i$ of $v_1,\dots,v_i$.

By the bracket-generating condition, $\pi_i$ are surjective and, by Lemma~\ref{l:mont1}, they induce an inner product space structure on $\distr_q^i/\distr_q^{i-1}$. Therefore, the nilpotentization of the distribution at $q$, namely 
\begin{equation}
\hat{M}_q = \distr_q \oplus \distr_q^2/\distr_q \oplus \ldots \oplus \distr_q^m/\distr_q^{m-1}\,,
\end{equation}
is an inner product space, as the orthogonal direct sum of a finite number of inner product spaces. As such, it is endowed with a canonical volume (defined up to a sign) $\omega_q \in (\wedge^n \hat{M}_q)^*$ such that $\omega_q(v_1,\ldots , v_n)=1$ for any orthonormal basis $v_1,\ldots,v_n$ of the vector space $\hat{M}_q$.

Finally, Popp's volume (computed at the point $q$) is obtained by transporting the volume of $\hat{M}_q$ to $T_q M$ through the map $\theta_q :\wedge^n T_q M \to \wedge^n \hat{M}_q$ defined in Lemma~\ref{l:mont2}. Namely 
\begin{equation}\label{eq:popppoint}
\popp_q =\theta_{q}^{*}(\omega_{q})= \omega_q \circ \theta_q,
\end{equation}
where $\theta_{q}^{*}$ denotes the dual map. Eq.~\eqref{eq:popppoint} is defined only in the domain of the chosen local frame. Since $M$ is orientable, with a standard argument, these local $n$-forms can be glued to a global one, called Popp's volume $\popp$. Moreover, Popp's volume is smooth, as follows, for example, by the explicit formula in \cite[Theorem 1]{nostropopp}. It is also clear that, in the Riemannian case, Popp's volume is the standard Riemannian one. Notice that Popp's volume is well defined only for equiregular structures.

\subsubsection{Behavior under isometries}

In the Riemannian setting, an isometry is a diffeomorphism such that its differential preserves the Riemannian scalar product. The concept is easily generalized to the sub-Riemannian case.
\begin{definition}
Let $(M,\distr_M,\metr_M)$ and $(N,\distr_N,\metr_N)$ be two (sub)-Riemannian structures. A (local) diffeomorphism $\phi: M \to N$ is a \emph{(local) isometry} if its differential $\phi_* : TM \to TN$ preserves the (sub)-Riemannian structure, namely (i) $\phi_* \distr_{M} = \distr_N$, and (ii) $\phi^* \metr_N = \metr_M$.
\end{definition}
The following is a trivial, but crucial property of Popp's volume that follows by its construction.
\begin{prop}\label{p:volumepres}
(Sub)-Riemannian (local) isometries $\phi:M \to M$  preserve Popp's volume, namely $\phi^*\popp = \popp$.
\end{prop}

\subsection{Intrinsic volumes}

An \emph{intrinsic definition of volume} is some algorithm that associates with any (sub)-Riemannian structure, a volume form, and this association is preserved by isometries. Let us be more precise.

\begin{definition}
An \emph{intrinsic definition of volume} is a map that associates, with any (sub)-Riemannian structure $(M,\distr,\g)$ a volume form $\omega_M$ on $M$ such that if $\phi: M \to N$ is a (sub)-Riemannian isometry between $(M,\distr_M,\g_M)$ and $(N,\distr_N,\g_N)$, then $\phi^* \omega_N  = \omega_M$.
\end{definition}
In the following we avoid this verbose terminology and we often talk about ``(intrinsic) volume'' to mean either the actual volume form $\omega_M$ or the map $M \mapsto \omega_M$.

\begin{example}
For any equiregular (sub)-Riemannian structure $(M,\distr,\metr)$, we can consider Popp's volume $\popp_M$. As a consequence of Proposition~\ref{p:volumepres}, this is an intrinsic definition of volume. Restricted to Riemannian structures, this definition of volume is the classical, Riemannian one.
\end{example}

Some intrinsic volumes are more intrinsic than others. Since the nilpotent approximation is the first order approximation of the (sub)-Riemannian structure at a point $q$, particularly simple definition of volumes are those that, loosely speaking, depend only on the metric invariants of the nilpotent approximation. 

We want to make the above statement more precise. To to this, we must go back to the very definition of nilpotent approximation. Recall that there is no canonical way to identify $T_q M$ with the nilpotentization $\hat{M}_q$. Still, as an immediate consequence of Lemma~\ref{l:mont2} we can identify the top wedge product of these vector spaces.

\begin{cor}
Let $(M,\distr,g)$ a (sub)-Riemannian manifold, $q \in M$ and $\hat{M}_q$ its nilpotent approximation at $q$. Then there exists a canonical map $\theta_q^* : (\wedge^n T_q M)^* \to (\wedge^n \hat{M}_q)^*$.
\end{cor}
\begin{proof}
This map is just the dual map of $\theta$ defined in the proof of Lemma~\ref{l:mont2}, where $E = T_q M$ and $F^i = \distr^i_q$.
\end{proof}
\begin{rmk}
In other terms the map $\theta_q$ canonically identifies parallelotopes living in $T_q M$ and those living in $\hat{M}_q = T_0 \hat{M}_q$. Notice that the (sub)-Riemannian metric does not play any role in the definition of $\theta_q$.
\end{rmk}

\begin{definition}
An intrinsic definition of volume $\omega$ is \emph{N-intrinsic} if, for any (sub)-Riemannian manifold $(M,\distr,\g)$ and any $q \in M$ the following diagram commutes
\begin{displaymath}
\begin{CD}
M @>{\omega}>> \omega_M(q) \\
@V{\mathrm{nil}_q}VV @VV{\theta^*_q}V\\
\hat{M}_q @>>\omega > \omega_{\hat{M}_q}(0) 
\end{CD}
\end{displaymath}
where $\mathrm{nil}_q$ associates with any (sub)-Riemannian manifold its nilpotent approximation at $q$.
\end{definition}

In other words, an intrinsic definition of volume is N-intrinsic if at any $q$ the volume form $\omega_M(q)$ agrees with the nilpotent volume form $\omega_{\hat{M}_q}(0)$ under the identification given by $\theta_q$.

\begin{example}
Let $(M,\distr,\g)$ be a Riemannian manifold ($\distr = TM$ and $\g$ is a Riemannian metric). The Riemannian volume on $M$ is the volume form $\omega_M$ such that, for any $q \in M$ and any oriented orthonormal parallelotope $v_1 \wedge \dots \wedge v_n$ at $q$ gives $\omega_M(q)(v_1 \wedge \dots \wedge v_n) = 1$.
We prove that this definition of volume is N-intrinsic. In this case, for any $q \in M$, the nilpotent approximation $\hat{M}_q = T_q M$ with the scalar product given by $\g_q$. Then the map $\theta_q$ is just the identity map between $T_q M$ and $\hat{M}_q$ (that are indeed the same vector space). Thus
\begin{equation}
\theta_q (v_1 \wedge \dots \wedge v_n) = v_1 \wedge \dots \wedge v_n,
\end{equation}
where the right hand side is seen as an element of $\wedge^n \hat{M}_q = \wedge^n (T_0 \hat{M}_q)$. Clearly $v_1 \wedge \dots \wedge v_n$ is an orthonormal parallelotope either when seen as an element of $\wedge^n (T_q M)$ or as an element of $\wedge^n \hat{M}_q$. By definition $\omega_M(q)(v_1 \wedge \dots \wedge v_n) =1$ and also $\omega_{\hat{M}_q}(0)(\theta (v_1 \wedge \dots \wedge v_n)) = 1$. Thus the Riemannian volume is N-intrinsic.
\end{example}

\begin{example}
Let $(M,\distr,g)$ be a Riemannian manifold ($\distr = TM$ and $g$ is a Riemannian metric). Let $\kappa : M \to \R$ be the scalar curvature function. We give a definition of volume $\omega$ as follows. For any $q \in M$ and any oriented normal parallelotope $v_1\wedge \dots \wedge v_n$ at $q$:
\begin{equation}
\omega_M(q)(v_1 \wedge \ldots \wedge v_n) = 1 + \kappa(q)^2.
\end{equation}
One can check that this is an intrinsic definition of volume (due to the fact that $\kappa$ is preserved by isometries). However, this is not N-intrinsic. In fact, without going into the details of the identification $\theta_q$ as above, we have that for any $q$, $\hat{M}_q$ is a flat Riemannian manifold (a vector space with inner product) thus $\hat{\kappa} = 0$. Then
\begin{equation}
\omega_{\hat M_q}(0)( \theta_q (v_1 \wedge \dots \wedge v_n)) = \omega_{\hat M_q}(0)(v_1 \wedge \ldots \wedge v_n) = 1 \neq 1+ \kappa(q)^2 = \omega_M(q)(v_1\wedge \dots \wedge v_n).
\end{equation}
\end{example}

\begin{example}
Popp's volume is N-intrinsic by construction. In fact it defined precisely by requiring the commutativity of the diagram above (see Section~\ref{s:Popp}).
\end{example}

\begin{example}
The spherical Hausdorff volume is N-intrinsic. In fact, in \cite{ABB-Hausdorff}, is is proved that the spherical Hausdorff volume is absolutely continuous w.r.t. any smooth volume form. In particular, the Radon-Nykodim derivative of the spherical Hausdorff volume w.r.t. Popp's volume is proportional to the left-invariant Haar measure of the unit ball in the nilpotent approximation. This function, in general, is not smooth, but only continuous.
\end{example}

We collect the results from these examples in the following propositions.
\begin{prop}\label{p:poppintrinsic}
For any (sub)-Riemannian structure $(M,\distr,\metr)$, let $\popp_M$ be Popp's volume. Then the definition of volume $M \mapsto \popp_M$ is N-intrinsic.
\end{prop}
\begin{prop}\label{p:riemannintrinsic}
For any Riemannian structure $(M,\g)$, let $\mathcal{R}_M$ be the Riemannian volume. Let $f \in C^\infty(M)$ be any function invariant by isometries. Then the definition of volume (for Riemannian manifolds) $M \mapsto f \mathcal{R}_M$ is N-intrinsic if and only if $f$ is constant.
\end{prop}
In other words, up to constant rescaling, the Riemannian one is the unique N-intrinsic definition of volume for Riemannian manifolds. This is due to the fact that the nilpotent approximation of Riemannian manifolds have no non-trivial metric invariants. In the sub-Riemannian case, this is not always true, due to the well-known presence of moduli for their nilpotent approximations. We close this section by discussing class of (sub)-Riemannian structures admitting a unique N-intrinsic definition of volume.

\subsubsection{Equi-nilpotentizable structures}

Equi-nilpotentizable (sub)-Riemannian structures have the same nilpotent approximation at any point.
\begin{definition}
A sub-Riemannian structure $(M,\distr,\g)$ is \emph{equi-nilpotentizabile} if for any $q,p \in M$ the nilpotent approximations $\hat{M}_q$ and $\hat{M}_p$ are isometric.
\end{definition}
The next theorem generalizes Proposition~\ref{p:riemannintrinsic} to equi-nilpotentizable structures (including Riemannian ones).
\begin{theorem}\label{t:equinilp}
Let $\omega$ and $\omega'$ two N-intrinsic definitions of volume. Then for any (sub)-Riemannian structure $(M,\distr,\g)$ there exists a smooth function $c_M :M \to \R$ such that $\omega_M = c_M \omega'_M$. The value of $c_M$ at $q$ depends only on the isometry class of the nilpotent approximation of the structure at $q$. In particular, for equi-nilpotentizable sub-Riemannian manifolds, any N-intrinsic definition of volume agrees up to a constant.
\end{theorem}
\begin{rmk}
By Proposition~\ref{p:poppintrinsic} Popp's one is an N-intrinsic definition of volume. Then, for equi-nilpotentizable (sub)-Riemannian structures, any N-intrinsic volume $M \mapsto \omega_M$ gives, up to a constant, Popp's one.\end{rmk}
\begin{proof}
Assume that there are two N-intrinsic definitions of volume $\omega$ and $\omega'$. Then, to any sub-Riemannian manifold $M$ we associate a smooth, never vanishing function:
\begin{equation}
c_M:= \omega_M/\omega'_M : M \to \R. 
\end{equation}
For any intrinsic definition of volume $\omega$ and isometry $\phi: M \to N$ of (sub)-Riemannian manifolds :
\begin{equation}
\phi^* \omega_N = \omega_M.
\end{equation}
Then, if we have two intrinsic definition of volumes $\omega$ and $\omega'$ we get, for isometric structures:
\begin{equation}\label{eq:isometric}
c_M = \frac{\omega_M}{\omega_M'} = \frac{\phi^* \omega_N}{\phi^* \omega_N'} = c_N \circ \phi.
\end{equation}
\begin{lemma}\label{l:constantf}
Let $(M,\distr,\g)$ be a left-invariant structure (see Definition~\ref{d:left-invariant}. Then $c_M$ is constant and depends only on the isometry class of $M$. Namely, if $(M,\distr_M,\g_M)$ and $(N,\distr_N,\g_N)$ are two isometric left-invariant structures, then $c_M = c_N$.
\end{lemma}
\begin{proof}
If in Eq.~\eqref{eq:isometric} we set $M=N$ and $\phi = L_q$, for $q \in M$, we get that $c_M$ is a left-invariant function, and then constant. Moreover, if $M$ and $N$ are two isometric left-invariant structures, then the two constants $c_M$ and $c_N$ must be equal.
\end{proof}

Let us go back to $(M,\distr,\g)$ and the two N-intrinsic definitions of volumes $\omega$ and $\omega'$. By definition of N-intrinsic, we have:
\begin{equation}
c_M(q) = \frac{\omega_M(q)}{\omega_M'(q)} = \frac{\omega_{\hat{M}_q}(0)}{\omega'_{\hat{M}_q}(0)} = c_{\hat{M}_q}(0), \qquad \forall q \in M.
\end{equation}
The structure on $\hat{M}_q$ is left-invariant. Then, by Lemma~\ref{l:constantf}, the function $c_{\hat{M}_q}$ depends only on the isometry class of $\hat{M}_q$ (which is the same for all $q \in M$). So $c_M$ is a constant that depends only on the isometry class of $\hat{M}_q$.
\end{proof}

\subsubsection{Homogeneous (sub)-Riemannian structures}
Another class of structures that admit a unique N-intrinsic definition of volume (actually, a unique intrinsic definition of volume) is the following.

\begin{definition}
A (sub)-Riemannian structure $(M,\distr,\metr)$ is \emph{homogenous} if the group $\mathrm{Iso}(M)$ of (sub)-Riemannian isometries of $M$ acts transitively.
\end{definition}

Indeed homogenous structures are equi-nilpotentizable, thus Theorem~\ref{t:equinilp} applies and any two N-intrinsic definition of volumes are proportional (and proportional to Popp's one). Still, something stronger holds. In fact, for these structures, Popp's one is the unique volume form (up to scaling) preserved by (sub)-Riemannian isometries (see \cite[Proposition 6]{nostropopp}). Then we have the following corollary.

\begin{cor}
For homogeneous (sub)-Riemannian structures, any two intrinsic (definitions of) volumes are N-intrinsic, and proportional to Popp's one.
\end{cor}

In particular this is true for left-invariant (sub)-Riemannian structures on Lie groups. 
\begin{definition}\label{d:left-invariant}
Let $M$ be a Lie group, and, for $q \in M$, let $L_q: M \to M$ denote the left multiplication. A (sub)-Riemannian structure $(M,\distr,\metr)$ is \emph{left-invariant} if $L_q : M \to M$ is an isometry for any $q \in M$.
\end{definition}
Since, on left-invariant structures, Popp volume is left-invariant we have the following.

\begin{cor}\label{c:popp=haar}
For left-invariant (sub)-Riemannian structures, any two intrinsic (definitions of) volumes are N-intrinsic, and proportional to Haar left-invariant volume.
\end{cor}


\section{The macroscopic Laplacian}\label{s:macro}

Let $(M,\distr,\metr)$ a sub-Riemannian structure and fix a volume form $\omega$ (not necessarily intrinsic or N-intrinsic). In this setting, we have a well defined notion of gradient and divergence (see Definitions~\ref{d:grad} and~\ref{d:diver}). The distribution is not assumed to be equiregular, unless one explicitly chooses $\omega = \popp$.

\begin{definition}
The \emph{macroscopic Laplacian} (w.r.t. $\omega$) is the differential operator
\begin{equation}
\Delta_\omega = \div_\omega\circ\grad.
\end{equation}
\end{definition}
This is a second order differential operator, symmetric on $C^\infty_0(M)$ w.r.t.\ the $L^2$ product induced by the measure $\omega$. A classical result by Strichartz (see \cite{strichartz,strichartzerrata}) is the following:
\begin{theorem}\label{t:selfad}
If $(M,\distr,\g)$ is complete as a metric space, then $\Delta_\omega$ is essentially self-adjoint on $C_0^\infty(M)$ and the associated heat operator is given by a positive, smooth, symmetric heat kernel.
\end{theorem}
We provide explicit formulas for $\Delta_\omega$, in terms of orthonormal frames. The proofs are routine computations.
\begin{lemma}
Let $X_1,\ldots,X_k$ be a (local) orthonormal frame for $\distr$ and $X_{k+1},\ldots,X_n$ some complement such that $X_1,\ldots,X_n$ is an oriented local frame for $TM$. Let $\theta \in C^\infty(M)$ defined by $\omega(X_1,\ldots,X_n) = e^\theta$. Then
\begin{equation}
\grad(f) = \sum_{i=1}^k X_i(f) X_i, \qquad f \in C^\infty(M),
\end{equation}
\begin{equation}
\div_\omega (X_i) = \sum_{\alpha=1}^n c_{\alpha i}^\alpha + d\theta (X_i), \qquad X \in \Gamma(TM).
\end{equation}
\end{lemma}

\begin{prop}
The macroscopic Laplacian w.r.t. the choice of $\omega$ is
\begin{equation}\label{eq:macroformula}
\Delta_{\omega} = \mathrm{div}_\omega \circ \grad = \sum_{i=1}^k X_i^2 + \mathrm{div}_\omega(X_i) X_i = \sum_{i=1}^k X_i^2 + \sum_{i=1}^k \sum_{\alpha=1}^n c_{\alpha i}^\alpha X_i + \grad(\theta).
\end{equation}
where the horizontal gradient $\grad(\theta)$ is seen as a derivation.
\end{prop}
Notice that, for any choice of $\omega$, the principal symbol of $\Delta_\omega$ is the Hamiltonian function $2H :T^*M \to \R$.

\begin{lemma}[Change of volume]\label{l:changeofvolume} Let $\omega$ be a volume form and $\omega' = e^g \omega$, with $g \in C^\infty(M)$. Then
\begin{equation}
\Delta_{\omega'} = \Delta_\omega + \grad (g),
\end{equation}
where $\grad (g)$ is meant as a derivation.
\end{lemma}
\begin{proof}
It follows from the change of volume formula for the divergence. In fact, for $X \in \Gamma(TM)$,
\begin{equation}
\div_{\omega'}(X) \omega' = \mathcal{L}_X \omega' = \mathcal{L}_X(e^g \omega) = X(g)e^g \omega + e^g \mathcal{L}_X \omega = (X(g) + \div_\omega(X))\omega'.
\end{equation}
Then $\div_{\omega'}(X) = \div_{\omega}(X) + X(g)$. The statement follows from the definition of $\Delta_\omega = \div_\omega \circ \grad$.
\end{proof}


\section{The microscopic Laplacian}\label{s-luca-formula}

Let $\V$ be a \emph{complement} (for the distribution $\distr$), namely a smooth sub-bundle of $TM$ such that 
\begin{equation}
T_qM = \distr_q \oplus \V_q, \qquad \forall q \in M.
\end{equation}
By duality we have an analogous splitting for the cotangent bundle, namely
\begin{equation}
T_q^*M = \mathbf{h}_q \oplus \mathbf{v}_q, \qquad \forall q \in M,
\end{equation}
where $\mathbf{h}$ (resp. $\mathbf{v}$) denotes the annihilator bundle of $\V$ (resp. of $\distr$). We call $\mathbf{v}$ the space of \emph{vertical covectors} and consequently $\mathbf{h}$ is the space of \emph{horizontal covectors} (see Fig.~\ref{f:verhor}). The Hamiltonian function is a fiber-wise non-negative quadratic form $H_q:=H|_{T_q^*M}$, with $\ker H_q = \mathbf{v}_q$. Then its restriction to $\mathbf{h}_q$ is positive definite, thus is a scalar product on it.
\begin{figure}
\centering
\begin{tikzpicture}[scale=0.9]
\node[draw, thick, trapezium,  trapezium left angle=75, trapezium right angle=-75, minimum height=3cm, minimum width=5cm,
trapezium stretches=true, shape border uses incircle,shape border rotate=0] at (0,0) {};
\draw[thick, -stealth] (0,0) -- (-1,3.5);
\node[draw, cylinder, minimum width=3cm, minimum height=5cm,aspect=4,,shape border rotate=90]   at (10,0) {};
\draw[thick, -stealth] (10,0) -- (10,3.5);
\draw[-triangle 60] (3,0) -- node[above] {``annihilator''} (6,0);
\draw[dashed, rotate around={12:(10,0)}] (10,0) ellipse (1.5cm and 10pt);
\draw[thick, xshift=0.3cm,yshift=0.2cm,rotate around={12:(6,-2)}] (6,-2) -- (12.5,-2) -- (14,0);
\draw[thick, xshift=0.3cm,yshift=0.2cm,rotate around={12:(6,-2)}] (14,0) -- (11.74,0);
\draw[thick, dashed, xshift=0.3cm,yshift=0.2cm,rotate around={12:(6,-2)}] (11.74,0) -- (8.7,0);
\draw[thick, xshift=0.3cm,yshift=0.2cm,rotate around={12:(6,-2)}] (8.7,0) -- (7.5,0) -- (6,-2);
\node at (-0.6,3.5) {$\V_q$};
\node[above] at (1.35,-1.5) {$\distr_q$};
\node at (13,3.5) {$T^*_qM$};
\node at (12,-1.5) {$\cyl_q$};
\node[right] at (11.7,1) {$\mathbf{h}_q = \V_q^\perp$};
\node[right] at (10,3.5) {$\mathbf{v}_q = \distr_q^\perp$};
\node at (2,3.5) {$T_qM$};
\end{tikzpicture}
\caption{Cotangent splitting $T_q^*M = \mathbf{v}_q \oplus \mathbf{h}_q$. Vertical covectors annihilate horizontal vectors ($\mathbf{v}_q = \distr_q^\perp$), and correspond to the trivial geodesic. Horizontal covectors annihilate the transverse vectors ($\mathbf{h}_q = \V^\perp_q$). They correspond to normal geodesics that are the ``most horizontal'' w.r.t. the choice of $\V$.}\label{f:verhor}
\end{figure}
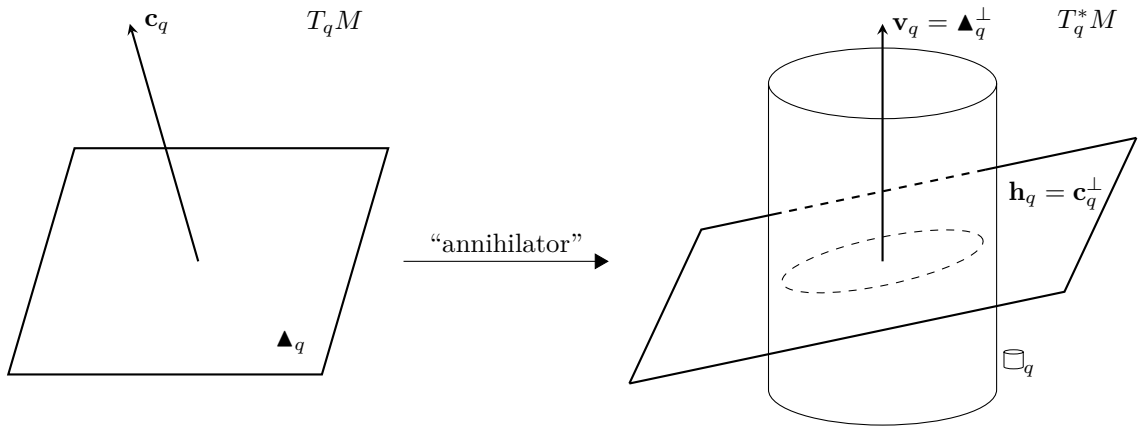
\begin{rmk}
An alternative definition, in the spirit of \cite{montgomerybook,GordLaeOlder,GordLae}, is the following. Extend $\metr$ to a Riemannian metric such that $\V$ and $\distr$ are orthogonal. Then we have a dual co-metric on $T^*M$. This co-metric depends on the choice of the orthogonal extension, but its restriction on $\mathbf{h}$ does not, and induces an Euclidean structure.
\end{rmk}
We obtain a $k-1$-dimensional sphere bundle over $M$ by taking the fiber-wise intersection $\cyl \cap \mathbf{h}$. In fact
\begin{equation}
\cyl_q \cap \mathbf{h}_q =\{\lambda \in T_q^*M \mid 2H(\lambda) =1,\; \lambda \in \mathbf{h}_q\}, \qquad \forall q \in M,
\end{equation}
is the Euclidean sphere $\mathbb{S}^{k-1}_q$ on the vector space $\mathbf{h}_q$ w.r.t. the scalar product $2H|_{\mathbf{h}_q}$.

\subsection{A class of adapted measures}
We define a class of probability measures on $\cyl_q$, obtained as the product of the standard probability measure on $\mathbb{S}^{k-1}_q$ and a probability measure on the complement $\mathbf{h}_q$. We are interested in a sufficiently general model, where \emph{all} geodesics on the cylinder potentially have a non-zero probability density. For this reason we include measures with non-compact support. This is indeed possible, but we need the probability to go to zero sufficiently fast at infinity along the ``non-compact directions'' of $\cyl_q = \mathbb{S}_q^{k-1} \times \mathbf{v}_q$.

\begin{definition}
Let $E$ be a vector space and $\alpha \in \mathbb{N}$. A Borel measure $\mu$ on $E$ is $\alpha$-\emph{decreasing} if any linear function $f : E \to \R$ belongs to $L^\alpha(E,\mu)$.
\end{definition}
If a measure is $\alpha$-\emph{decreasing} then it is $\beta$-\emph{decreasing} for any $\beta \leq \alpha$. A compactly supported probability measure is $\alpha$-decreasing for all $\alpha$. Finally, we notice that one needs to check the condition only for a complete set of linear projections (the functions $(x_1,\ldots,x_m) \mapsto x_m$ in terms of some basis). 

\begin{definition}\label{d:decreasing}
Let $\mu_q$ be a probability measure on $\cyl_q = \mathbb{S}^{k-1}_q \times \mathbf{v}_q$. We say that the collection of measures $\{\mu_q\}_{q \in M}$ is \emph{adapted to the splitting} if there exists a $2$-decreasing probability measure $\mu_{\mathbf{v}_q}$ on $\mathbf{v}_q$ such that
\begin{equation}
\mu_q = \mu_{\mathbb{S}_q^{k-1}} \times \mu_{\mathbf{v}_q}, \qquad \forall q \in M,
\end{equation}
where $\mu_{\mathbb{S}_q^{k-1}}$ denotes the standard uniform measure on $\mathbb{S}_q^{k-1} = \cyl_q \cap \mathbf{h}_q$. Moreover, we assume that if  $f: T^*M \to \R$ is a continuous and fiberwise linear function, the map $q \mapsto \int_{\cyl_q} f^2 \mu_q$ (which is well defined) is continuous.
\end{definition}
\begin{rmk}
As we will see, the choice of $\alpha=2$ is forced by the properties of the exponential map. In the Riemannian case, there is no need of this, since the fibers of the unit tangent bundle are compact. The regularity assumption is needed for the uniform convergence in the next definition.
\end{rmk}

For the next definition recall that, by Lemma~\ref{l:prolong}, there exists $\eps_0 >0$ such that all arc-length parametrized geodesics $\gamma_\lambda(t) = \exp_q(t,\lambda)$, with $\lambda \in \cyl_q$, are well defined for $t \in [0,\eps_0)$.
\begin{definition}\label{Def:MainOperator}
Consider a splitting $TM =\distr \oplus \V$ and some choice of adapted measure $\{\mu_q\}_{q \in M}$. The \emph{microscopic Laplacian} is the differential operator:
\begin{equation}
(L^\V \phi)(q):=\lim_{t\to 0^+} \frac{2 k}{ t^2}\int_{\cyl_q} (\phi(\exp_q(t,\lambda))-\phi(q))\mu_q(\lambda), \qquad \forall q \in M.
\end{equation}
\end{definition}
Surprisingly, this definition does not depend on the choice of the measure adapted to the splitting, and justifies the notation $L^\V$. Moreover, the right hand side converges uniformly on compact sets for $t \to 0$. The proof of these facts are contained in the proof of Theorem~\ref{t:microformula}.
\begin{rmk}
As discussed in Section~\ref{s:intro}, this is the operator associated with the limit of a random walk, where, at each step, normal geodesics are chosen on $\cyl_q$ with the given probability measure $\mu_q$. Our construction generalizes the one given in \cite{GordLaeOlder,GordLae}. The latter can be recovered by setting
\begin{equation}
\mu_q = \mu_{\mathbb{S}^{k-1}_q} \times \delta_{\mathbf{v}_q}, \qquad \forall q \in M,
\end{equation}
where $\delta_{\mathbf{v}_q}$ is a Dirac mass centered at $0 \in \mathbf{v}$. This is certainly an $\alpha$-decreasing measure $\forall\alpha\geq 0$, and satisfies the regularity assumptions of definition~\ref{d:decreasing}. With this choice, at each step of the associated random walk one moves only on a $k$-dimensional submanifold.
\end{rmk}
\begin{theorem}\label{t:microformula}
The microscopic sub-Laplacian $L^\V$ depends only on the choice of the complement $\V$, and the convergence of~\ref{Def:MainOperator} is uniform on compact sets. Moreover, in terms of a local orthonormal frame $X_1,\ldots,X_k$ for $\distr$ and a local frame $X_{k+1},\ldots,X_n$ for $\V$, we have:
\begin{equation}\label{eq:microformula}
L^\V = \sum_{i=1}^k X_i^2 + \sum_{i,j=1}^k c_{ji}^j X_i,
\end{equation}
where $c_{ij}^k \in C^\infty(M)$ are the structural functions associated with the given frame.
\end{theorem}
\begin{proof}
Let $\nu_1,\ldots,\nu_n$ be the co-frame dual to $X_1,\ldots,X_n$:
\begin{align}
\distr = & \spn\{X_1,\ldots,X_k\}, & \mathbf{h} = & \spn\{\nu_1,\ldots,\nu_k\},\\ 
\V = & \spn\{X_{k+1},\ldots,X_n\}, & \mathbf{v} = & \spn\{\nu_{k+1},\ldots,\nu_{n}\}.
\end{align}
Fix $q \in M$. In coordinates $(h_1,\ldots,h_n):T_q^*M \to \R^n$ induced by the choice of the frame, we have:
\begin{equation}
\cyl_{q} := \cyl \cap T_{q}^*M = \{(h_1,\ldots,h_k,h_{k+1},\ldots,h_n) \mid h_1^2+\ldots+ h_k^2 = 1\}  \simeq \mathbb{S}^{k-1} \times \R^{n-k}.
\end{equation}
Since $q$ is fixed, we dub $\mathbf{v}_q = \R^{n-k}$ and $\mu_{\mathbf{v}_q} = \mu_{n-k}$. By plugging in the Taylor expansion of Lemma~\ref{l:taylor} in the definition of $L^\V$ we get:
\begin{equation}
(L^\V \phi)(q) := \lim_{t \to 0^+} \frac{2k}{t^2} \int_{\cyl_q}\left\lbrace t h_i X_i(\phi) + \frac{1}{2}t^2\left[h_j c_{ji}^\alpha h_\alpha X_i(\phi) + h_i h_j X_j(X_i(\phi))\right] + t^3 r_\lambda(t) \right\rbrace \mu(\lambda),
\end{equation}
where repeated indices are summed (according to the generalized Einsten's convention) and we suppressed some of the explicit evaluations at $q$. We compute the three terms of the integrand.

1. The linear in $t$ term vanishes. In fact, the integrand depends only on the variables $h_1,\ldots,h_k$, and
\begin{equation}
\int_{\cyl_q}  h_i X_i(\phi) \mu(\lambda) =  X_i(\phi) \int_{\R^{n-k}}  \mu_{n-k} \int_{\mathbb{S}^{n-k}} h_i \mu_{\mathbb{S}^{k-1}} = 0,
\end{equation}
since the integral of any linear function on the sphere vanishes.

2. For the quadratic in $t$ term, we distinguish two contributions:
\begin{equation}
h_j c_{ji}^\alpha h_\alpha X_i(\phi) + h_i h_j X_j(X_i(\phi)) =  \underbrace{ h_j c_{ji}^{\bar\ell} h_{\bar\ell} X_i(\phi)}_{1} + \underbrace{ h_j c_{ji}^\ell h_\ell X_i(\phi) + h_i h_j X_j(X_i(\phi))}_{2},
\end{equation}
where, we recall that the range of summation of barred latin indices is from $k+1$ to $n$. The first term, not present in the Riemannian case, is the product of a linear function on $\mathbb{S}^{k-1}$ and a linear function on $\mathbb{R}^{n-k}$. Since the measure $\mu$ is a product, by Fubini, we can first perform the integral on $\mathbb{S}^{k-1}$:
\begin{equation}
\int_{\cyl_q}  h_j c_{ji}^{\bar\ell} h_{\bar\ell} X_i(\phi) = c_{ji}^{\bar\ell}X_i(\phi)\int_{\mathbb{S}^{k-1}} h_j \mu_{\mathbb{S}^{k-1}} \int_{\R^{n-k}}\mu_{n-k} = 0.
\end{equation}
On the other hand, the second term does not contain any $h_{\bar\ell}$, for $\bar\ell=k+1,\ldots,n$ and is the restriction on $\mathbb{S}^{k-1}$ of a quadratic form $Q = Q_{ij}h_i h_j$, with:
\begin{equation}
Q_{ij} = X_i(X_j(\phi)) + c_{i\ell}^j X_\ell(\phi), \qquad i,j=1,\ldots,k.
\end{equation}
Recall that for any quadratic form $Q : \R^k \to \R$ we have
\begin{equation}
\int_{\mathbb{S}^{k-1}} Q(v) \mu_{\mathbb{S}^{k-1}}(v) =  \frac{1}{k} \tr(Q).
\end{equation}
Thus we get, for the quadratic in $t$ term, the following expression:
\begin{equation}
k\int_{\cyl_q}  Q_{ij}h_i h_j \mu =  k \int_{\R^{n-k}} \mu_{n-k} \int_{\mathbb{S}^{k-1}} Q_{ij}h_i h_j \mu_{\mathbb{S}^{k-1}}  =  \tr(Q) = \sum_{i=1}^k X_i^2(\phi) + \sum_{i,\ell =1}^k c_{i\ell}^i X_\ell(\phi).
\end{equation}
where we restored the explicit summations in the last term.

3. The final term to compute is the remainder. By Lemma~\ref{l:taylor}, if $t$ is sufficiently small, the remainder $r_\lambda(t)$ of the Taylor's expansion is uniformly bounded by a quadratic polynomial in the unbounded variables $h_{\bar\imath}$. 
\begin{equation}
\left\lvert\frac{k}{t}^2\int_{\cyl_q} r_\lambda(t) t^3\right\rvert \leq k t \int_{\R^{n-k}} \left(A +B_{\bar\ell}|h_{\bar\ell}| + C_{\bar\imath\bar\jmath}|h_{\bar\imath}||h_{\bar\jmath}|\right)\mu_{n-k}\int_{\mathbb{S}^{k-1}} \mu_{\mathbb{S}^{k-1}}, \qquad \forall t \leq \eps_0.
\end{equation}
By Lemma~\ref{l:taylor} this estimate holds uniformly on an open ball $B(q,\eps_0)$. Then, for $q$ in a compact set $K$:
\begin{equation}
\left\lvert\frac{k}{t}^2\int_{\cyl_q} r_\lambda(t) t^3\right\rvert\leq t D \sum_{\bar\ell = k+1}^n \int_{\cyl_q}  (h_{\bar\ell})^2 \mu_q = t g(q).
\end{equation}
where $D$ is a constant and, by our assumption on the family of measures, $g(q)$ is continuous (recall that the $h_{\alpha}$ are precisely the linear, smooth functions $\lambda \mapsto \langle \lambda, X_\alpha\rangle$ on $T^*M$). Then the remainder goes to zero uniformly on $K$. The fact that $L^\V$ depends only on the choice of the complement (and not on the adapted measure) is a consequence of Eq.~\eqref{eq:microformula}.
\end{proof}

\begin{rmk}
We stress that equiregularity is not assumed in Theorem~\ref{t:microformula}, for Lemma~\ref{l:taylor} is completely general.
\end{rmk}

\subsection{Horizontal divergence}\label{horizontal-div}

We can write the operator $L^\V$ in a more classical fashion, introducing the so-called \emph{horizontal divergence}. As the classical divergence depends on the choice of a volume form (or a density), the horizontal one depends on the choice of the complement $\V$.

\begin{definition}
Let $m \in \mathbb{N}$. The bundle of \emph{horizontal $m$-alternating tensors} on $M$ is the disjoint union
\begin{equation}
\Lambda_\distr^m M := \bigsqcup_{q\in M} (\wedge^m \distr_q)^*.
\end{equation}
where $(\wedge^m \distr_q)^*$ is the space of $m$-alternating functionals on $\distr_q$. Sections of $\Lambda_\distr^m M$ are called \emph{horizontal $m$-forms}.
\end{definition}
Any $m$-form induces an horizontal $m$-form by restriction. Given the complement $\V$, the fiber-wise linear projection $\pi_\distr : TM \to \distr$ is well defined, and for any horizontal $m$-form $\eta$, we can define a $m$-form $\tilde{\eta}:= \eta \circ \pi_\distr$.
\begin{definition}
Given an horizontal $m$-form $\eta$, and a vector field $X \in \Gamma(TM)$, we define the \emph{horizontal Lie derivative} $\mathcal{L}_X^\V\eta$ as the horizontal $m$-form such that its value at $q \in M$ is
\begin{equation}
(\mathcal{L}_X^\V \eta)_q = \left.\frac{d}{dt}\right\rvert_{t=0}(P_t^* \eta \circ \pi_\distr)_q,
\end{equation}
where $P_t$ is the flow of $X$.
\end{definition}
This is the same definition of the Lie derivative of $m$-forms, with the addition of $\pi_\distr$, needed here since $\eta$ acts only on $m$-tuples of horizontal vectors. One obtains readily:
\begin{equation}\label{eq:hordivformula}
(\mathcal{L}^\V_X \eta) (Y_1,\ldots,Y_m):= X(\eta(Y_1,\ldots,Y_m)) - \sum_{i=1}^m \eta(Y_m,\ldots,\pi_\distr [X,Y_i],\ldots,Y_m),
\end{equation}
where $Y_1,\ldots,Y_m \in \Gamma(\distr) $. If $\distr$ is oriented (as a vector bundle), with $\rank \distr = k$, we define a canonical horizontal $k$-form $\eta$ such that
\begin{equation}
\eta(X_1,\ldots,X_k) = 1,
\end{equation}
for any oriented orthonormal frame $X_1,\ldots,X_k  \in \Gamma(\distr) $.
\begin{definition}
Let $X \in \Gamma(TM)$, and let $\eta$ be a non-zero $k$ form. The \emph{horizontal divergence} $\div^\V(X)$ is the function defined by
\begin{equation}
\mathcal{L}_X^\V \eta = \div^\V(X) \eta.
\end{equation}
\end{definition}
\begin{rmk}
The definition does not depend on the choice of orientation. A general definition for non-orientable distributions can be given in a standard way with horizontal densities (projectivized non-vanishing $k$ forms).
\end{rmk}
Indeed $\div^\V(X)$ computes the infinitesimal change of volume, under the flow of $X$, of the projection on $\distr$ of the standard parallelotope of $\distr$. As a direct consequence of Eq.~\eqref{eq:microformula} and Eq.~\eqref{eq:hordivformula}, we obtain the following.
\begin{prop}
Let $TM = \distr \oplus \V$. The microscopic Laplacian is
\begin{equation}
L^\V = \div^\V \circ \grad = \sum_{i=1}^k X_i^2 + \div^\V(X_i) X_i,
\end{equation}
where $X_1,\ldots,X_k$ is an orthonormal frame, $\grad$ is the horizontal gradient and $\div^\V$ is the horizontal divergence.
\end{prop}
This formula must be  compared directly with Eq~\eqref{eq:macroformula} for the macroscopic operator, where the horizontal divergence $\div^\V$ is replaced by the divergence $\div_\omega$ w.r.t. some volume form $\omega$ on $M$.

\section{The equivalence problem}\label{s:equiv}

Fix a volume form $\omega$ and a complement $\V$. We consider both $L^\V$ and $\Delta_\omega$ as second order differential operators on the space of compactly smooth supported functions $C^\infty_0(M)$, with the $L^2$ product:
\begin{equation}
(f_1,f_2)_\omega = \int_M f_1 f_2  \omega, \qquad f_1,f_2 \in C^\infty_0(M).
\end{equation}
The operator $\Delta_\omega$ is symmetric by construction. What about $L^\V$? The principal symbol of both operators, as a function on the cotangent bundle, is twice the Hamiltonian $2H:T^*M \to \R$.

\begin{theorem}\label{t:compatibility}
Let $\V$ be a complement and $\omega$ a volume form. Then $L^\V$ is symmetric w.r.t.\ the $L^2$ product induced by $\omega$ if and only if $L^\V = \Delta_\omega$ or, equivalently, if and only if
\begin{equation}
\chi^{(\V,\omega)} := \Delta_\omega -L^\V= \sum_{i=1}^k \sum_{j=k+1}^n c_{ji}^j X_i +\grad( \theta)= 0,
\end{equation}
where $\theta = \log |\omega(X_1,\ldots,X_n)|$ and $c_{ij}^\ell \in C^\infty(M)$ are the structural functions associated with an orthonormal frame $X_1,\ldots,X_k$ for $\distr$ and a frame $X_{k+1},\ldots,X_n$ for $\V$.
\end{theorem}

\begin{proof}
To get the explicit formula for $\chi^{(\V,\omega)}$ compare Eq.~\eqref{eq:macroformula} with Eq.~\eqref{eq:microformula}. For what concerns the symmetry, suppose that $L^\V$ is symmetric w.r.t. some choice of $\omega$. Since $\Delta_\omega$ is symmetric, $\chi^{(\V,\omega)}$ is also symmetric, that is (suppressing some notation):
\begin{equation}
(\chi(f_1),f_2)_\omega = (f_1,\chi(f_2))_\omega, \qquad \forall f_1,f_2 \in C^\infty_0(M).
\end{equation}
If we choose $f_2 = 1$ on some domain larger than the support of $f_1$ this is equivalent to $\int_M \chi(f_1) \omega = 0$ for any $f_1 \in C^\infty_0$. In particular $\chi =0$. The converse is clear.
\end{proof}
The condition $\chi^{(\V,\omega)} = 0$ is equivalent to the following system of PDEs:
\begin{equation}
X_i(\theta) + \sum_{\bar\jmath = k+1}^n c_{\bar\jmath i}^{\bar\jmath} = 0, \qquad i=1,\ldots,k.
\end{equation}
In particular $L^\V$ is symmetric w.r.t. some volume $\omega$ if and only if, for fixed $\V$, the above system admits a global solution $\theta$. If this is the case, the associated volume is given by $\omega(X_1,\ldots,X_n) = e^\theta$.
\begin{rmk}
The \emph{compatibility condition} $\chi^{(\V,\omega)}= 0$ is the same appearing in \cite[Theorem 5.13]{GordLae}, written in a different form. We call \emph{compatible} the pairs $(\V,\omega)$ solving the compatibility condition.
\end{rmk}

Finally, if $L^\V$ is symmetric w.r.t. some choice of volume $\omega$, the latter is unique (up to constant rescaling).
\begin{lemma}\label{l:uniqueness}
If $L^\V = \Delta_\omega$, then $L^\V = \Delta_{\omega'}$ if and only if $\omega = c \omega'$, where $c$ is a non-zero constant.
\end{lemma}
\begin{proof}
Let $\omega' = e^g \omega$ with $g \in C^\infty(M)$. Then, using the change of volume formula (see Lemma~\ref{l:changeofvolume}):
\begin{equation}
\chi^{(\V,\omega')} = \Delta_{\omega'} - L^\V = \Delta_{\omega} +\grad(g) - L^\V = \chi^{(\V,\omega)} + \grad(g).
\end{equation}
Thus both pairs are compatible iff $\grad (g) = 0$, that is (by the bracket generating condition) iff $g$ is constant.
\end{proof}


\section{Carnot groups}\label{s:carnot}

A Carnot group $G$ of step $m$ is a simply connected Lie group whose Lie algebra of left-invariant vector fields $\mathfrak{g}$ admits a nilpotent stratification of step $m$, namely
\begin{equation}
\mathfrak{g} = \mathfrak{g}_1 \oplus \dots\oplus \mathfrak{g}_m, \qquad \mathfrak{g}_i \neq \{0\}, \qquad \forall i=1,\ldots,m,
\end{equation}
with
\begin{equation}	
[\mathfrak{g}_1,\mathfrak{g}_j] = \mathfrak{g}_{1+j}, \qquad \forall 1\leq j\leq m-1, \qquad \text{and} \qquad \mathfrak{g}_{m+1} = \{0\}.
\end{equation}
A left-invariant sub-Riemannian structure on $G$ is obtained by defining a scalar product on $\mathfrak{g}_1$ or, equivalently, by declaring a set $X_1,\ldots,X_{k} \in \mathfrak{g}_1$ a global orthonormal frame. In particular, $\distr|_q = \mathfrak{g}_1|_q$, for all $q \in G$. The group exponential map,
\begin{equation}
\mathrm{exp}_{G} : \mathfrak{g} \to G,
\end{equation}
associates with $v \in \mathfrak{g}$ the element $\gamma(1)$, where $\gamma: [0,1] \to G$ is the unique integral line of the vector field $v$ such that $\gamma(0) = 0$. Since $G$ is simply connected and $\mathfrak{g}$ is nilpotent, $\mathrm{exp}_G$ is a smooth diffeomorphism. Thus we identify $G \simeq \R^m$, endowed with a polynomial product law.
The adjoint endomorphism $\mathrm{ad}_{X}:\mathfrak{g} \to \mathfrak{g}$ is:
\begin{equation}
\mathrm{ad}_X (Y) := [X,Y], \qquad \forall X,Y \in \mathfrak{g}.
\end{equation}
Notice that, if $X \in \mathfrak{g}_j$, then 
\begin{equation}
\ad_X|_{\mathfrak{g}_\ell} : \mathfrak{g}_\ell \to \mathfrak{g}_{\ell+j}, \qquad \forall \ell,j=1,\ldots,m,
\end{equation}
by the graded structure of $\mathfrak{g}$. 
\begin{rmk}
In the literature, these structures are also referred to as \emph{Carnot groups} of type $(k,n)$, where $k = \dim \distr = \rank \mathfrak{g}_1$ is the \emph{rank} of the distribution and $n$ is the dimension of $G$. 
\end{rmk}

By definition, Carnot groups are left-invariant sub-Riemannian structures (see Definition~\ref{d:left-invariant}). Then it is natural to fix a left-invariant volume, that is proportional to Popp's one $\popp$ by Corollary~\ref{c:popp=haar}. Moreover, we restrict to left-invariant complements $\V$. In this setting, we rewrite the compatibility condition $\chi^{(\V,\popp)} = 0$ in a more invariant fashion. To do this, we choose left invariant orthonormal frames $X_1,\ldots,X_k$ for $\distr$ and left-invariant frames $X_{k+1},\ldots,X_n$ for $\V$. Thanks to the splitting $\distr \oplus \V$ we have the projections $\pi_\V$ and $\pi_\distr$ on $TM$.

\begin{lemma}
For Carnot groups, the compatibility condition for left-invariant volumes and complements is
\begin{equation}
\tr(\pi_\V \circ \ad_{X_i}) = 0, \qquad \forall i =1,\ldots,k.
\end{equation}
\end{lemma}
\begin{proof}
We rewrite the compatibility condition as
\begin{equation}
\sum_{j=k+1}^{n} c_{ji}^j X_i +\cancel{X_i(\theta)} = \sum_{j=k+1}^n \nu^j([X_j,X_i]) = - \tr(\pi_\V\circ \mathrm{ad}_{X_i}\circ \pi_\V ) = -\tr(\pi_\V \circ \ad_{X_i}) = 0,
\end{equation}
where we used the cyclic property of the trace and the fact that projections are idempotent. The function $\theta$ appearing in the compatibility condition is constant, since both $\popp$ and $X_1,\ldots,X_n$ are left-invariant.
\end{proof}

On Carnot groups, thanks to the canonical identification $\mathfrak{g}_j = \distr^j/\distr^{j-1}$, there is a natural left-invariant Riemannian extension. In fact, it is sufficient to consider the orthogonal direct sum
\begin{equation}
\mathfrak{g} = \distr \oplus \mathfrak{g}_2 \oplus \ldots\oplus \mathfrak{g}_m,
\end{equation}
where on each factor $\mathfrak{g}_j$, with $g\geq 2$, we have a well defined scalar product induced by the maps $\pi_j$ of Eq.~\eqref{eq:pimap}. This gives a smooth left-invariant scalar product on $TM$. Popp's volume $\popp$ is the Riemannian volume of such natural Riemannian extension. We can address questions \textbf{Q2} and \textbf{Q3} of Section~\ref{s:intro}, restricted to the class of left-invariant volumes and complements.

The most natural complement comes from the very definition of the Carnot structure:	
\begin{equation}
\V_0 := \mathfrak{g}_2 \oplus \dots\oplus \mathfrak{g}_m,
\end{equation}
\begin{prop}
For any Carnot group $G$, we have $\Delta_{\popp} = L^{\V_0}$.
\end{prop}
\begin{proof}
By definition of $\V_0$ and the stratified structure of the Carnot algebra, we have
\begin{equation}
\pi_{\V_0}\circ \mathrm{ad}_{X_i}\circ \pi_{\V_0} = \mathrm{ad}_{X_i}|_{\V_0},
\end{equation}
that is indeed nilpotent, since $\ad_{X_i}$ is, and $\tr(\ad_{X_i}) =0$. Then the compatibility condition is satisfied.
\end{proof}

Any distinct left-invariant complement is the graph of a non-trivial linear map $\ell: \V_0 \to \distr$, that is
\begin{equation}
\V_\ell := \{X+ \ell( X)\mid X \in \V_0\}.
\end{equation}
\begin{prop}\label{p:nonuniquenessCarnot}
If $\Delta_\popp = L^\V$, then $\V = \V_\ell$, with
\begin{equation}
\tr(\ell \circ \ad_{X_i}) = 0, \qquad \forall i=1,\ldots,k.
\end{equation}
\end{prop}
\begin{proof}
We use the shorthand $\pi_\ell$ (resp. $\pi_0$) for the two projections $\pi_{\V_\ell}$ and $\pi_{\V_0}$ respectively. Indeed $\pi_\ell = \pi_0 + \ell \circ \pi_0$. The new complement is compatible iff, for all $i=1,\ldots,k$
\begin{equation}
0 = \tr(\pi_\ell \circ \ad_{X_i} \circ \pi_\ell) = \tr(\pi_\ell^2 \circ \mathrm{ad}_{X_i}) = \tr(\pi_\ell \circ \mathrm{ad}_{X_i}) = \tr(\ell \circ \pi_0 \circ \ad_{X_i}) = \tr(\ell \circ \ad_{X_i}),
\end{equation}
where we used the cyclic property of the trace, the fact that projectors are idempotent, that $\V_0$ satisfies the compatibility condition and the fact that the image of $\ad$ is contained in $\V_0$.
\end{proof}
Any $\ell$ such that $\ell(\mathfrak{g}_2) = 0$ gives an example where the compatible complement is non-unique. This class is always non-empty for any Carnot group of step $m \geq 3$, but is trivial for step $2$. 
\begin{example}
We provide an example of non-uniqueness for corank 1 Carnot groups. Choose a structure with
\begin{equation}
\mathfrak{g}_1 = \spn\{X_1,\ldots,X_k\}, \qquad \mathfrak{g}_2 = \spn\{X_0\},
\end{equation}
\begin{equation}
[X_i,X_j] = A_{ij} X_0, \qquad i,j=1,\ldots,k
\end{equation}
for a skew symmetric matrix $A \in \mathfrak{so}(k)$ (the singular values of $A$ define uniquely any Carnot group of type $(k,k+1)$, up to isometries, see for example \cite{LR-Enumerative}). Any other compatible complement is of the form:
\begin{equation}
\V = \spn\left\lbrace X_0 + \sum_{j=1}^k \ell_j X_j\right\rbrace, \qquad \text{with} \qquad \ell \in \ker A.
\end{equation}
Then the number of distinct compatible complements is equal to $\dim \ker A$. In particular, the only case in which we have uniqueness is when $\ker A = \{0\}$, that is a contact Carnot group.
\end{example}


\section{Corank $1$ structures}\label{s:corank1}

Consider a sub-Riemannian structure $(M,\distr,\metr)$ with $\rank \distr= \dim M -1$. We assume that $\distr$ has step $2$, that is $\distr_q^2 = T_qM$ for all $q \in M$. This is a popular class of structures including contact, quasi-contact, and more degenerate ones. Locally $\distr = \ker \eta$ for some one-form $\eta$. The endomorphism $J_\eta: \Gamma(\distr) \to \Gamma(\distr)$ is defined by 
\begin{equation}
g(X,JY) = d\eta(X,Y), \qquad \forall X,Y \in \Gamma(\distr).
\end{equation}
It is skew-adjoint w.r.t. the sub-Riemannian scalar product, namely $J^* = -J$. Moreover, $J \neq 0$. In fact, if this were the case, a simple application of Cartan's formula gives, for any local frame $X_1,\ldots,X_k$ of $\distr$:
\begin{equation}
\eta([X_i,X_j]) = -d\eta(X_i,X_j) + \cancel{X_i(\eta(X_j))}-\cancel{X_j(\eta(X_i))} = g(JX_j,X_i) = 0,
\end{equation}
in contradiction with the step $2$ condition. It follows that $\tr(JJ^*) > 0$. This observation leads to the following construction. If $g \in C^\infty(M)$, then $g \eta$ gives the same distribution. On the other hand one can check that if $\eta' = g\eta$, then $d\eta'|_{\distr} = g d\eta|_{\distr}$, that is $J_{\eta'} = g J_\eta$. We can fix $\eta$, up to a sign, with the condition
\begin{equation}
\| J\|^2 = \tr(JJ^*) =  \sum_{i,j=1}^k g(X_i,JX_j)^2 = 1.
\end{equation}
In this case, we say that the local one-form $\eta$ is \emph{normalized}. The existence of a global, normalized one-form depends on the distribution. In particular there exists a unique (up to a sign), normalized global one-form if and only if $\distr$ is co-orientable (i.e. the quotient bundle $TM/\distr$ is orientable or, equivalently, trivial).
\begin{rmk}
If $TM/\distr$ is not orientable, a global normalized one-form does not exist. This is the case, for example, for the structure on $M = \R^2 \times \R P^1$ with
\begin{equation}
\distr:=\ker(\sin\theta dx - \cos\theta dy),
\end{equation}
where we identify $\R P^1 \simeq \R/\pi\mathbb{Z}$ with the coordinate $\theta$. One still could work with a global object $[\eta]$, section of the bundle whose fibers are $(T^*_q M\setminus\{0\})/\mathbb{Z}_2$. In particular the value $[\eta](q)$ is the equivalence class $\{\pm\eta(q)\}$.  All the results and formulas appearing in the following do not depend on the choice of the representative of the class at each point. For this reason, it is not a restriction to assume that $\eta$ is a well-defined global one-form.
\end{rmk}

We investigate the relation $\Delta_\omega = L^\V$. We first rewrite the compatibility condition $\chi^{(\V,\omega)}=0$ in this case. Observe that any complement is (locally) generated by a vector field $X_0$, that we normalize with $\eta(X_0)=1$.
\begin{lemma}\label{l:compatibility-corank1}
Let $\V = \spn\{X_0\}$, with $\eta(X_0)=1$. The compatibility equation $\chi^{(\omega,\V)} = 0$ is
\begin{equation}
\sum_{i=1}^k d\eta(X_0,X_i) X_i = \grad(\theta) \qquad \text{with} \qquad \theta = \log|\omega(X_1,\ldots,X_k,X_0)|.
\end{equation}
\end{lemma}
\begin{proof}
In the local adapted frame $X_1,\ldots,X_k,X_0$, the compatibility equation $\chi^{(\omega,\V)} = 0$ is
\begin{equation}
\sum_{i=1}^k c_{0i}^0  X_i + \grad(\theta) = 0.
\end{equation}
This simple form is due to the fact that the structure has corank $1$. Then we rewrite
\begin{equation}
c_{0i}^0 = \eta([X_0,X_i]) = - d\eta(X_0,X_i) +\cancel{X_0(\eta(X_i))} - \cancel{X_i(\eta(X_0))},
\end{equation}
where we used Cartan's formula and the normalization $\eta(X_0) = 1$.
\end{proof}
The existence and uniqueness of compatible complements in this case depends on the dimension of $\ker J$.
\begin{prop}\label{p:volumeuniquecomplement}
Fix a volume $\omega$, and let $p:=\dim\ker J$. Then:
\begin{itemize}
\item if $p =0$, then $\exists!$ complement $\V$ such that $L^\V = \Delta_\omega$;
\item if $p>0$, the space of $\V$ such that $L^\V = \Delta_\omega$ is either empty, or an affine space over $\ker J$.
\end{itemize}
\end{prop}
\begin{proof}
Let $p=0$, we check that the equation $\chi^{(\omega,\V)} = 0$ has a unique solution for fixed $\omega$. Assume $\V = \spn\{Z + \xi\}$, where $Z$ is transverse to $\distr$ such that $\eta(Z)=1$ and $\xi \in \Gamma(\distr)$. Then, by Lemma~\ref{l:compatibility-corank1}, the compatibility condition gives
\begin{equation}
\grad(\theta) = d\eta(Z+\xi,X_i)X_i =d\eta(Z,X_i)X_i+ g(\xi,JX_i)X_i =d\eta(Z,X_i)X_i -J\xi,
\end{equation}
There is a unique $\xi$ satisfying this equation, namely $\xi = -J^{-1}[\grad(\theta)+d\eta(Z,X_i)X_i]$. Now let $p >0$, and suppose that $X_0$ and $X_0'$ are two different (normalized) generators for complements $\V$ and $\V'$, compatible with the volume $\omega$, with $\eta(X_0)=\eta(X_0')=1$. The normalization implies $X_0-X_0' \in \distr$. According to Lemma~\ref{l:compatibility-corank1}
\begin{equation}
\sum_{i=1}^k d\eta(X_0,X_i)X_i = \grad(\theta), \qquad\text{and}\qquad \sum_{i=1}^k d\eta(X_0',X_i)X_i =\grad(\theta)',
\end{equation}
with $\theta' = \log| \omega(X_1,\ldots,X_k,X_0')|$. Since $X_0-X_0' \in \distr$, 
\begin{equation}
\theta = \log| \omega(X_1,\ldots,X_k,X_0)| = \log |\omega(X_1,\ldots,X_k,X_0')| = \theta'.
\end{equation}
Then we have $d\eta(X_0-X_0',X_i) = 0$, for all $i =1,\ldots,k$. This implies $X_0-X_0' \in \ker J$. Conversely, one can check that for any (normalized) generator $X_0$ of a compatible complement $\V$, any other complement generated by $X_0 \oplus \ker J$ is compatible with the same volume $\omega$.
\end{proof}

\subsection{On natural Riemannian extensions and Popp's volume}

We end this section with a general result about volumes of Riemannian extensions of corank $1$ structures.

\begin{prop}\label{p:naturalriemannian}
Let $(M,\distr,\g)$ a corank $1$ sub-Riemannian structure with normalized one-form $\eta$. Take any local vector field $Z$ such that $\eta(Z) =1$. Consider the (local) Riemannian extension of $(M,\distr,\g)$ obtained by declaring $Z$ an orthonormal vector field. Then the Riemannian volume of this Riemannian extension does not depend on the choice of $Z$ and is equal to Popp's volume.
\end{prop}
\begin{proof}
This statement is an application of the explicit formula of \cite[Proposition 16]{nostropopp}, stated in the contact case, whose proof holds unchanged for corank $1$ sub-Riemannian structures (it only requires that $\tr(JJ^*) \neq 0$).
\end{proof}
Proposition~\ref{p:naturalriemannian} gives another reason for considering Popp's volume in the corank $1$ case (besides the fact, valid in general, that Popp's volume is N-intrinsic). We now specialize Proposition~\ref{p:volumeuniquecomplement} to the contact and quasi-contact cases, giving an answer to {\bf Q2} and {\bf Q3} of Section~\ref{s:intro}.

\section{Contact structures}\label{s:contact}

Let $M$ be a smooth manifold with $\dim M = 2d +1$. A corank sub-Riemannian structure is \emph{contact} if $\ker J = 0$. This is the least degenerate case. Since $M$ is odd dimensional, there exists a unique vector $Z$, called \emph{Reeb vector field}, such that $d\eta(Z,\cdot) = 0$ and $\eta(Z) =1$.

The next corollaries follow from (the proof of) Proposition~\ref{p:volumeuniquecomplement} and give positive answers to {\bf Q2} and {\bf Q3} mentioned in Section~\ref{s:intro}.
\begin{cor}
For any volume $\omega$ there exists a unique complement $\V$ such that $L^\V = \Delta_\omega$. This complement is generated by the vector
\begin{equation}
X_0 = \Z -J^{-1}\grad(\theta),
\end{equation}
where $\theta = \log|\omega(X_1,\ldots,X_k,\Z)|$.
\end{cor}

Contact sub-Riemannian structures have a natural Riemannian extension, obtained by declaring $\Z$ a unit vector orthonormal to $\distr$. By Proposition~\ref{p:naturalriemannian}, Popp's volume is the Riemannian volume of this extension.
\begin{cor}\label{c:reebpopp}
Let $\popp$ be the Popp's volume. The unique complement $\V$ such that $L^{\V} = \Delta_{\popp}$ is generated by the Reeb vector field. Moreover, $\popp$ is the unique volume (up to constant rescaling) with this property.
\end{cor}
\begin{proof}
By \cite[Rmk. 4]{nostropopp}, $\theta=1$. Then $\grad (\theta )=0$ and $X_0 = \Z$. The uniqueness follows from Lemma~\ref{l:uniqueness}.
\end{proof}

\subsection{Integrability conditions}\label{s:integrability}
The inverse problem, namely for a fixed complement $\V$, find a volume $\omega$ such that $L^\V = \Delta_\omega$ is more complicated (and, in general, has no solution). In the contact case we find an explicit integrability condition.
\begin{prop}\label{p:integrability}
Let $\V = \spn\{X_0\}$. Define the one-form $\alpha := \frac{i_{X_0} d\eta}{\eta(X_0)}$ and the function $g=\frac{d\alpha \wedge\eta \wedge (d\eta)^{d-1}}{\eta\wedge (d\eta)^d }$. Then there exists a volume $\omega$ such that $L^{\V} = \Delta_\omega$ if and only if
\begin{equation}
d\alpha - dg \wedge \eta - g d\eta = 0.
\end{equation}
In this case, $\omega$ is unique up to constant rescaling.
\begin{rmk}
If $\V$ is the Reeb direction, $\alpha = 0$, the integrability condition is satisfied and we recover Corollary~\ref{c:reebpopp}.
\end{rmk}
\end{prop}
\begin{proof}
We assume $\V = \spn\{\Z + \xi\}$ for some $\xi \in \Gamma(\distr)$. This is equivalent to normalizing $X_0$ in such a way that $\eta(X_0) =1$. A volume $\omega$ is uniquely specified by the function $\theta = \log|\omega(X_1,\ldots,X_k,\Z)|$. The same proof of Proposition~\ref{p:volumeuniquecomplement} leads to the compatibility condition
\begin{equation}
-J \xi = \grad(\theta).
\end{equation}
We have to solve the following problem: given an horizontal vector field $X$, find a function $\theta \in C^\infty(M)$ such that $X = \grad(\theta) $. With $X$ we can associate a one-form $\alpha$ such that
\begin{equation}
\alpha(Y) = g(X,Y), \qquad \forall Y \in \Gamma(\distr).
\end{equation}
For our case, $X=-J\xi$. Then $\alpha(Y) = g(-J\xi,Y) = -g(Y,J\xi) = -d\eta(Y,\xi) = i_\xi Y = i_{X_0} d\eta (Y)$ and in this case we set $\alpha :=i_{X_0} d\eta$. In the language of forms, the above problem is equivalent to $\alpha|_{\distr} = d\theta|_{\distr}$, and has solution iff
\begin{equation}
\alpha + g\eta = d \theta, \qquad \text{ for some } g \in C^\infty(M).
\end{equation}
A (local) necessary and sufficient condition for the existence of such $\theta$ is that there exists a closed representative in the class $\alpha + g\eta$. Namely $g$ must satisfy
\begin{equation}\label{eq:compat}
d \alpha + dg\wedge \eta + g d\eta = 0.
\end{equation}
If such a $g$ exists, is uniquely expressed in terms of $\alpha$. Taking the appropriate wedge product with $\eta$, and then with $d\eta$ $d-1$ times, we get
\begin{equation}
g \eta\wedge (d\eta)^d = - \eta \wedge d\alpha \wedge (d\eta)^{d-1} \qquad \Rightarrow \qquad g = -\frac{d\alpha \wedge\eta \wedge (d\eta)^{d-1}}{\eta\wedge (d\eta)^d }.
\end{equation}
The $n$-form at the denominator is never zero, it is equivalent to the non-degeneracy assumption $\ker J =\{0\}$. We have to check that such a $g$ gives a closed representative, i.e. that $f$ solves Eq.~\eqref{eq:compat}. Plugging the explicit expression of $g$ into Eq.~\eqref{eq:compat} gives the condition of our statement. Uniqueness follows from Lemma~\ref{l:uniqueness}.
\end{proof}
\begin{rmk}
In the three dimensional case $d=1$ and the integrability condition is simplified. Fix a basis $X_1,X_2$. We can assume $d\eta(X_1,X_2) = 1$ and we get for the function $g$
\begin{equation}
g = -\frac{d\alpha \wedge\eta \wedge (d\eta)^{d-1}}{\eta\wedge (d\eta)^d } = -\frac{d\alpha \wedge\eta(X_1,X_2,\Z)}{\eta\wedge d\eta (X_1,X_2,\Z)} = - d\alpha(X_1,X_2).
\end{equation}
One can check that, when restricted to $\distr$, Eq.~\eqref{eq:compat} is always satisfied. Then the only non-trivial condition is given by taking the contraction with the Reeb field: $i_\Z d\alpha + d(d\alpha(X_1,X_2))= 0$.
\end{rmk}

\section{Quasi-contact structures} \label{s:quasicontact}

Let $M$ a smooth manifold with $\dim M =2d+2$. A corank $1$ sub-Riemannian structure $(M,\distr,\metr)$ is \emph{quasi-contact} if $\distr_0:=\ker J$ has positive dimension. We assume \emph{minimal degeneration}, that is $\dim \distr_0 = 1$.

\subsection{The quasi-Reeb vector field}\label{s:quasireeb} 
We discuss the construction of a ``canonical'' vector transverse to $\distr$, analogous to the Reeb vector field, in the quasi-contact case. We learned this construction from \cite{gregquasi} (where a different normalization is used).

In the construction that follows we always assume that $J_q :\distr_q \to \distr_q$ has distinct eigenvalues at all points on $M$. This is true for a generic quasi-contact structure and outside a codimension 3 closed stratified subset. Thus, $J$ has distinct imaginary eigenvalues $\pm i \lambda_j$ with 
\begin{equation}
0= \lambda_0 < \lambda_1 < \dots < \lambda_d \in \R.
\end{equation}
where $\lambda_i$ are smooth functions on $M$. The normalization condition gives $2\sum_{i=1}^d \lambda_i^2 =1$. 

For any $j$, denote by $\distr_j \subset \distr$ the $2$-dimensional (real) eigenspace associated with the eigenvalues $\pm i\lambda_j$ (we use the same notation for the eigenspaces and the associated sub-bundle). By definition of $J$, the $\distr_j$ are mutually orthogonal, for all $i=0,1,\ldots,d$. We can choose generators $X_j,Y_j$ of each distribution $\distr_j$ by taking smooth families of (generalized) eigenvectors of $J$. Namely
\begin{equation}
J (X_j+ i Y_j) = i \lambda_j (X_j+iY_j), \qquad \forall j =1,\ldots,d.
\end{equation}
Equivalently
\begin{equation}
J X_j = - \lambda_j Y_j, \qquad J Y_j = \lambda_j X_j, \qquad \forall j=1,\ldots,d.
\end{equation}
Notice that $\lambda_j g(X_j,Y_j) =  g(J Y_j,Y_j) = 0$ since $J$ is skew-symmetric. Then $X_j$ and $Y_j$ are orthogonal. We choose them to be orthonormal. Moreover, let $W \in \ker J$ of unit norm. In particular, $X_1,Y_1,\ldots,X_d,Y_d,W$ is a local orthonormal frame for $\distr$.
\begin{rmk}
In terms of the orthonormal frame $X_1,Y_1,\ldots,X_d,Y_d,W \in \Gamma(\distr)$, the matrix representing $J$ is
\begin{equation}
J \simeq \begin{pmatrix}
\lambda_1 K & & & \\
& \ddots & &\\
& & \lambda_d K&\\
& & & 0
\end{pmatrix}, \qquad K: = \begin{pmatrix}
0 & 1 \\
-1& 0
\end{pmatrix}.
\end{equation}
\end{rmk}

Choose $j \in \{1,\ldots,d \}$. Let $\distr_j^2:=[\distr_j,\distr_j]+\distr_j$ be the distribution generated by taking all possible brackets of sections of $\distr_j$ of length up to $2$. Notice that $[\distr_j,\distr_j]$ is transverse to $\distr$. In fact, for any $V,W \in \Gamma(\distr_j)$:
\begin{equation}
\eta([V,W]) = -d\eta(V,W) + \cancel{V(\eta(W))} - \cancel{W(\eta(V))} = g(W,JV)
\end{equation}
and we conclude using the fact that $J|_{\distr_j}$ is non-degenerate. Then $[\distr_j,\distr_j]$ is a three dimensional distribution, generated by the orthonormal vector fields $X_j,Y_j$ defined above and their bracket $[X_j,Y_j]$.
\begin{definition}\label{d:quasireeb}
The \emph{quasi-Reeb} vector field $\Z_j$ is the unique vector field such that
\begin{equation}
\Z_j \in [\distr_j,\distr_j], \qquad d\eta(\Z_j,\distr_j) = 0, \qquad \eta(\Z_j) = 1.
\end{equation}
\end{definition}
\begin{prop}\label{p:quasireebformula}
In terms of the orthonormal generators $X_j,Y_j$ of $\distr_j$, we have 
\begin{equation}
\Z_j = - \frac{1}{\lambda_j}[X_j,Y_j]  +\frac{d\eta([X_j,Y_j],Y_j)}{\lambda^2_j} X_j -\frac{d\eta([X_j,Y_j],X_j)}{\lambda^2_j} Y_j.
\end{equation}
\end{prop}
\begin{proof}
It's a linear algebra computation with Cartan's formula. By the first condition
\begin{equation}
\Z_j = \alpha X_j + \beta Y_j + \gamma [X_j,Y_j], \qquad \alpha,\beta,\gamma \in C^\infty(M).
\end{equation}
Using the second condition with $X_j \in \distr_j$ we get
\begin{equation}
0 = d\eta (\Z_j,X_j) = \beta d\eta(Y_j,X_j) + \gamma d\eta([X_j,Y_j],X_j), \qquad \Rightarrow  \qquad \beta = \gamma\frac{d\eta([X_j,Y_j],X_j)}{\lambda_j}.
\end{equation}
Using the second condition with $Y_j \in \distr_j$ we get
\begin{equation}
0 = d\eta (\Z_j,Y_j) = \alpha d\eta(X_j,Y_j) + \gamma d\eta([X_j,Y_j],Y_j), \qquad \Rightarrow  \qquad \alpha = -\gamma\frac{d\eta([X_j,Y_j],Y_j)}{\lambda_j}.
\end{equation}
We conclude the proof using the third condition (normalization), which gives
\[
1 = \eta(\Z_j) = \gamma\eta([X_j,Y_j]) = -\gamma d\eta(X_j,Y_j) = \gamma g(Y_j,JX_j) = -\gamma \lambda_j. \qedhere
\]
\end{proof}
\begin{rmk}
If the structure is nilpotent (of step $2$), any Lie bracket of length greater than $2$ vanishes, and by the explicit formula above we have
\begin{equation}
\Z_j = - \frac{1}{\lambda_j}[X_j,Y_j].
\end{equation}
If the structure is also a Carnot group, $[X_j,Y_j]$ belongs to the second stratum $\mathfrak{g}_2$, for any $j$. Since there is a unique vector field in the second stratum such that $\eta(\Z_j) = 1$, it follows that $\Z_j$ is the same for all $j \in \{1,\ldots,d\}$. Then for quasi-contact Carnot groups this construction is canonical and does not depend on the choice of $j$.
\end{rmk}
\begin{rmk}
At the points on the manifold where the eigenvalues of $J$ cross, the quasi-Reeb vector fields $\Z_j$ are no longer well defined. However, the volume of the Riemannian extensions obtained by declaring $\Z_j$ a unit vector orthogonal to $\distr$, can be extended to a well defined, global smooth volume. In fact, as a consequence of Proposition~\ref{p:naturalriemannian}, the Riemannian volume of any one of these extensions coincides with Popp's volume (as $\eta(\Z_j) = 1$), and the latter is clearly a well defined volume form on the whole manifold.
\end{rmk}
\begin{rmk}
The construction above defines $d$ different vector fields, defined on the regions of the manifold where the associated eigenvalue $i\lambda_j$ has algebraic multiplicity equal to $1$. In some cases one can choose a canonical, smooth $\Z_j$: for example when the smallest eigenvalue has globally minimal multiplicity. For the case $d=1$ there is always such a choice. In general, one could define a unique natural transverse vector by taking the average $\Z:=\frac{1}{d}\sum_{j} \Z_j$. Its regularity properties, however, are not clear when the the spectrum of $J$ is not simple.
\end{rmk}

\subsection{An example of non-existence}

Proposition~\ref{p:volumeuniquecomplement} states that, in the contact case (i.e. when $J$ is non-degenerate), for any volume $\omega$ there exists a unique complement $\V$. In the quasi-contact case the situation changes dramatically. Compatible complements are never unique as soon as $\ker J \neq \{0\}$. Surprisingly, compatible complements might not exist. We discuss an example where for a given volume (Popp's one) there are \emph{no} compatible complements.

\begin{example}
Consider the quasi-contact structure on $M =\mathbb{R}^4$, with coordinates $(x,y,z,w)$, defined by
\begin{equation}
\eta = \frac{g}{\sqrt{2}} \left(dw - \frac{y}{2} dx+\frac{x}{2} dy\right),
\end{equation}
where $g$ is any monotone, strictly positive function (e.g. $g = e^z$).
The metric is defined by the following global orthonormal frame for $\distr$:
\begin{equation}
X = \frac{1}{\sqrt{g}}\left( \partial_x + \frac{1}{2}y \partial_w \right), \qquad Y = \frac{1}{\sqrt{g}}\left(\partial_y - \frac{1}{2} x \partial_w\right), \qquad \Z = \frac{1}{\sqrt{g}} \partial_z.
\end{equation}
This is essentially $\mathbb{H}_3 \oplus \mathbb{R}$ with a scaled metric. One can check that
\begin{equation}
d \eta = \frac{\dot{g}}{\sqrt{2}} \left(dz \wedge dw + \frac{y}{2} dx \wedge dz -\frac{x}{2} dy \wedge dz \right)  + \frac{g}{\sqrt{2}}  dx \wedge dy,
\end{equation}
where the $\dot{g} = \partial_z g$. The representative matrix $A$ of $d\eta$ in coordinates $(x,y,z,w)$ is
\begin{equation}
A \simeq  \frac{1}{\sqrt{2}}\begin{pmatrix}
0 & g & \tfrac{y}{2}\dot{g} & 0 \\
-g & 0 & -\tfrac{x}{2} \dot{g} & 0 \\
-\tfrac{y}{2} \dot{g} & \tfrac{x}{2}\dot{g} & 0 & \dot{g} \\
0 & 0 & -\dot{g} & 0
\end{pmatrix}, \qquad \text{with} \qquad \det A = \frac{ g^2 \dot{g}^2}{4} > 0.
\end{equation}
In particular, $\ker d\eta =\{0\}$. Moreover, one can check that
\begin{equation}
J X = -\frac{1}{\sqrt{2}}Y, \qquad J Y = \frac{1}{\sqrt{2}} X, \qquad J \Z = 0.
\end{equation}
Thus we have the correct normalization
\begin{equation}
\|J\|^2 = \tr(JJ^*) = 1.
\end{equation}
Choose $\omega = \popp$, the Popp's volume. We look for a complement $\V = \spn\{X_0\}$ such that $L^\V = \Delta_\omega$. We can assume, rescaling $X_0$, that $\eta(X_0) = 1$. Using the explicit formula from \cite{nostropopp}, we have
\begin{equation}
\theta_{\popp} = \log| \popp(X_1,\ldots,X_k,X_0)| = \frac{1}{\sqrt{\sum_{i,j=1}^k d\eta(X_i,X_j)^2}} = \frac{1}{\|J\|} = 1.
\end{equation}
We look for solutions of the equation $\chi^{(\popp,\V)} = 0$. Since $\theta_{\popp} = 1$, we get from Lemma~\ref{l:compatibility-corank1}
\begin{equation}
d\eta(X_0,X_i) = 0, \qquad \forall i =1,\ldots,k.
\end{equation}
This implies $X_0 \in \ker d\eta$, but $d\eta$ has trivial kernel.
\end{example}
\begin{rmk}
It can be shown that $L^\V \neq \Delta_\popp$ is a generic property for quasi-contact structures. The precise statement and the proof involves techniques from transversality theory and is out of the scope of this paper.
\end{rmk}

This result in the case $d=1$ (that is $\dim M =4$, the lowest dimension for quasi-contact structures) is particularly surprising, for the following reason. The nilpotent approximation of any corank $1$ sub-Riemannian structure is a corank $1$ Carnot group. This is a sub-Riemannian structure on $\R^n$ that, in coordinates $(x,z) \in \R^{n-1} \times \R$ is generated by the global orthonormal frame:
\begin{equation}\label{eq:frameC}
X_i = \frac{\partial}{\partial x_i} - \frac{1}{2}\sum_{j=1}^{n-1}A_{ij} x_j \frac{\partial}{\partial z}, \qquad i =1,\ldots,n-1,
\end{equation}
where $A \in \mathfrak{so}(n-1)$. Isometry classes of corank $1$ Carnot groups are determined by the string $0\leq \alpha_1 < \ldots < \alpha_p$ of the $p$ distinct non-negative singular values of $A$, up to multiplication for a global constant, and their geometric multiplicities (see \cite{LR-Enumerative} for corank $1$ structures, and \cite[Rmk. 1]{ALG-path} for higher corank). In particular, for $n =4$, $A \in \mathfrak{so}(3)$ there is a unique such a string: $(0,1)$ where $0$ has multiplicity $1$ and $1$ has multiplicity $2$. In particular, any corank $1$ Carnot group in dimension $n=4$ is isometric to the one defined by the frame~\eqref{eq:frameC}, with:
\begin{equation}
A = \begin{pmatrix}
0 & 1 & 0 \\
-1 & 0 & 0 \\
0 & 0 & 0
\end{pmatrix}.
\end{equation}
It follows that any corank $1$ sub-Riemannian in dimension $4$ is quasi-contact and equi-nil\-po\-tentizable. In this case, we proved that there exists a unique N-intrinsic volume (up to scaling) and is given by Popp's one, and a unique N-intrinsic Laplacian $\Delta_\popp$ (see Theorem~\ref{t:equinilp}). In our example, this unique N-intrinsic Laplacian has \emph{no} compatible complement or, in other words, the macroscopic diffusion operator has no microscopic counterpart.


\section{Convergence of random walks}\label{a:randomwalk}

We let $M$ be a smooth, geodesically complete (sub)-Riemannian manifold. Our goal here is to give fairly general conditions for a sequence of random walks on $M$ to converge. In particular, we will work with a larger class of random walks than those treated elsewhere in the paper, which will require some differences in notation. In this section we work in the general (sub)-Riemannian framework, so that geodesics are specified by the choice of an initial covector $\lambda \in T^*M$ and the (sub)-Riemannian distribution has rank $k \leq n = \dim M$.

\subsection{A general class of random walks}

By a random walk, in this section, we mean a process $X_t$ that travels along a geodesic (with constant speed) for a fixed length of time (say $\hh>0$), at which time a new geodesic is chosen at random (and independently of the particle's past), which is then traveled for time $\hh$, and this procedure repeats indefinitely. The result is that $X_0, X_{\hh},X_{2\hh},\ldots$ is a Markov chain on $M$, and for $t\in(i\hh,(i+1)\hh)$, $X_t$ interpolates between $X_{i\hh}$ and $X_{(i+1)\hh}$ along a geodesic between them. We distinguish this from the related scheme of following a randomly chosen geodesic for a random, exponentially distributed length of time, at which point a new geodesic is chosen and a new (independent) exponential clock is started, and so on, which we will refer to as a random flight. (The logic being that a walk involves regular steps, and one does not turn midstep, whereas a flight seems to evoke a mode of travel which can change direction at any time.) A random flight has the nice property that, because the exponential is memory-less, it is a time-homogeneous process (assuming the choice of geodesic is time-homogeneous, of course).

More concretely, for $\hh>0$, we have a family of probability measures $\mu^{\hh}_q$ on the cotangent spaces $T^*_qM$. Then $X^{\hh}_{\hh(i+1)} = \exp_{X^{\hh}_{\hh i}}(h,\lambda/h)$ where $\lambda$ is a covector chosen according to $\mu^{\hh}_{X^{\hh}_{\hh i}}$ independently of the previous steps, and the path travels from $X^{\hh}_{\hh i}$ to $X^{\hh}_{\hh(i+1)}$ along the geodesic determined by $\lambda$, at constant speed given by $\|\lambda\|/h$ (and for a distance given by $\|\lambda\|$). That is, for $t\in [\hh i,\hh(i+1)]$, we have $X^{\hh}_t = \exp_{X^{\hh}_{\hh i}}\lp (t-\hh i),\lambda/\hh \rp$.

Note that, unlike earlier in the paper, here we index our random walks by the size of the time-step, instead of by the size of the spacial-step. This is because we no longer require that $\mu^{\hh}_q$ be supported on covectors of fixed length. Indeed, earlier the paper, we have considered the case when $\mu_q$ is a probability measure on covectors of length 1, and $\mu^{\hh}_q$ is given by parabolic scaling of $\mu_q$; that is, if $f_{\hh}$ is the fiber-preserving map that takes a covector $\lambda$ to $\sqrt{2\hh  k} \lambda$, then $\mu^{\hh}_q$ is the pushforward of $\mu_q$ by $f_{\hh}$. (Also compare the present set-up with Definition \ref{Def:MainOperator}.) However, here we work in the generality of random walks as just described.

We are interested in the convergence of a sequence of random walks $X^{\hh}_t$ (on $M$) as $\hh\rightarrow0$ to a (continuous time) diffusion. A key role is played by the following associated operator (which we have already seen). For any $\phi\in C_0^{\infty}(M)$, let
\[
L_{\hh} \phi(q) = \frac{1}{\hh}\lp\E\lb \left . \phi\lp X^{\hh}_{\hh}\rp \right| X^{\hh}_0=q\rb -\phi(q)\rp = 
  \frac{1}{\hh}\lp\int_{T^*_qM} \phi\lp \exp_q(h,\lambda/h)\rp  \mu^{\hh}_q(\lambda)-\phi(q) \rp .
\]
The idea is that, under suitable assumptions, the random walks will converge to the diffusion generated by $\lim_{\hh\rightarrow0}L_{\hh}$ as $\hh\rightarrow0$. That the behavior in the limit is governed only by the second-order operator, and not any other features of the random walks, should be seen as a version of Donsker invariance.

\subsection{Background on diffusions and pathspace}

Next, we recall some basic facts about the diffusions which will arise as limits of our random walks. In what follows, we assume that $L$ is a smooth second-order operator with non-negative-definite principal symbol and without zeroth-order term. This means that, in local coordinates $x_1,\ldots,x_n$, $L$ can be written as
\[
\sum_{i,j=1}^n a_{ij}\partial_{x_i}\partial_{x_j} + \sum_{i=1}^n b_i \partial_{x_i},
\]
where the $a_{ij}$ and $b_i$ are smooth functions, and the matrix $\lb a_{ij}\rb$ is symmetric and non-negative definite. Alternatively, $L$ can be locally written as
\[
Y_0 + \sum_{i=1}^n Y_i^2
\]
for smooth vector fields $Y_0,Y_1,\ldots, Y_n$. Recall that, in this case, there exists a unique diffusion generated by $L$, although this diffusion may explode in finite time (that is, with positive probability, a path might exit every compact subset of $M$ in finite time). For simplicity, we assume that the diffusion generated by $L$ does not explode (indeed, our analysis is fundamentally local). Let $P_q$ be the measure on $\Omega_M$ (the space of continuous paths from $[0,\infty)$ to $M$) corresponding to this diffusion (starting from $q$). Then $P_q$ is uniquely characterized by the fact that
\[
\phi(\omega_t)-\int_0^t L \phi\lp \omega_s\rp \, ds -\phi(q) \quad\text{is a $P_q$-martingale with $P_q\lp \omega_0=q\rp=1$}.
\]

A flexible and powerful approach to the convergence of random walks, based on martingale theory, was essentially already provided by Stroock and Varadhan in \cite{SAndV} (specifically, in Section 11.2). However, they develop their approach in the setting of Euclidean space. To generalize to the Riemannian or sub-Riemannian setting is relatively straight-forward. The main things one needs to do are to replace the systematic use of Euclidean coordinates by the use of coordinate-free or local coordinate-based expressions and to replace the linear interpolation from $X_{i\hh}$ and $X_{(i+1)\hh}$ via the appropriate geodesic segment. In regard to the local versus global nature of the problem, in principle we could even allow for the limiting diffusion to explode in finite time, as just described. However, this requires a more substantial reworking of the underlying framework ($M$ should be compactified by adding a point at infinity-- the Alexandrov one-point compactification, the space of continuous paths on $M$ should be suitably modified, etc.), which would be a substantial digression from the main thrust of this work. Indeed, we are most interested in understanding various choices of sub-Laplacian on a sub-Riemannian manifold, and this is really a local problem. Thus, there is no harm in restricting our attention to the situation when the underlying diffusion on $M$ does not explode.

Following \cite[Section 1.3]{SAndV}, for $\omega\in\Omega_M$, let $\omega_t$ be the position of $\omega$ at time $t$. We define a metric on $\Omega_M$ by
\begin{equation}\label{Eqn:PathspaceDist}
d_{\Omega_M}(\omega,\tilde{\omega}) = \sum_{i=1}^{\infty} \frac{1}{2^i}\frac{\sup_{0\leq t\leq i}d_M(\omega_t,\tilde{\omega}_t)}{1+ \sup_{0\leq t\leq i}d_M(\omega_t,\tilde{\omega}_t)} .
\end{equation}
This metric makes $\Omega_M$ into a Polish space, and the notion of convergence that it induces is that of uniform convergence on bounded time intervals. We equip $\Omega_M$ with its Borel $\sigma$-algebra $\sM$ and give it the natural filtration $\sM_t$ generated by $\{\omega_s:0\leq s\leq t\}$. For probability measures on $\Omega_M$, we will always work with respect to the weak topology (in probabilists' terminology; this is the weak* topology in the language of functional analysis, corresponding to convergence against bounded, continuous test functions). We will generally be interested in a (Markov) family of probability measures, indexed by points of $M$. (Indeed, we have already seen that our random walks correspond to such families.) In that case, we will write them as $P_q$ for $q\in M$.

The case of $M=\bR^n$ (with the standard Euclidean metric) is the one considered by Stroock and Varadhan. Rather than generalize their approach directly (which would perhaps be more mathematically satisfying, but would also take longer), we exploit a standard trick to make the general case a special case of the Euclidean case. In particular, first note that the sub-Riemannian metric on $M$ can be extended to a Riemannian metric (not necessarily in any canonical way, but that doesn't matter here). Then the Nash embedding theorem implies that $M$ can be isometrically embedded into some Euclidean space of high enough dimension, say $\bR^N$ (and this embedding is smooth, although not necessarily proper if $M$ is not compact). 

\subsection{Control of the interpolating paths}

Assuming that the operators $L_{\hh}$ converge to an operator $L$ of the type just discussed, in the sense that, for any $\phi\in C_0^{\infty}(M)$, we have $L_{\hh} \phi\rightarrow L\phi$ uniformly on compacts as $\hh\rightarrow0$, is almost the only condition that we need. (Indeed, for $M=\bR^N$ this is enough.) However, the convergence of the $L_{\hh}$ depends only on the positions of the random walks at the sequences of times $0,\hh, 2\hh,\ldots$, and does not necessarily control the path traveled between these times. To give a simple example of what can go wrong, consider $M=\bS^1$ with the standard metric. If $X^{\hh}_0=q$, let the walk travel a full circle either clockwise or counter-clockwise, each with probability $1/2$, at constant speed in time $\hh$. In other words, if $\theta$ is the standard coordinate, we let $\mu^{\hh}_q = \pm 2\pi d\theta$ where the sign has probability $1/2$ of being positive or negative. Then $L^{\hh}$ is the zero operator for all $\hh$ (since $q=X^{\hh}_0=X^{\hh}_{\hh}=X^{\hh}_{2\hh}=\cdots$), but it's clear from Eq.~\eqref{Eqn:PathspaceDist} that the walks are not converging. This is in contrast to Euclidean space, where the entire path $X^{\hh}_t$ for $t\in[i\hh,(i+1)\hh]$ is determined by $X^{\hh}_{i\hh}$ and $X^{\hh}_{(i+1)\hh}$. (Indeed, in \cite{SAndV}, each step of the random walk is determined by choosing a point in $\bR^N$ according to, in slightly modified notation, a probability measure $\Pi^h_q$ on $\bR^N$. Then the path of the particle during this step is given by linear interpolation. The appropriate geometric generalization of this is our measure $\mu_q^h$ on the co-tangent space, which we can think of as giving the next step in the walk along with a path for the particle to travel to get there.)

Thus, we will also need to assume that for any $\rho>0$, any compact $Q\subset M$, and any $\alpha>0$, there exists $\hh_0>0$ such that
\begin{equation}\label{Eqn:PathRegularity}
\frac{1}{\hh}P^{\hh}_{q}\lb \sup_{0\leq s\leq\hh} d_M\lp X^{\hh}_{0} , X^{\hh}_{s}\rp \leq \rho\rb > 1-\alpha
\end{equation}
whenever $q\in Q$ and $\hh<\hh_0$. This assumption is not especially restrictive. First of all, on Euclidean space, Lemmas 11.2.1 and 11.2.2 of \cite{SAndV} show that if $L_{\hh} \phi\rightarrow L\phi$ uniformly on compacts as $\hh\rightarrow0$ for any $\phi\in C_0^{\infty}(\bR^N)$, then, under some global uniformity assumptions on the $\mu^{\hh}_q$, for any $\rho>0$ we have
\[
\lim_{\hh\rightarrow0} \sup_{\substack{q\in \bR^N \\ i\in\{0,1,2,\ldots\}}} \frac{1}{\hh}P^{\hh}_{q}\lb d_{\bR^N}\lp X^{\hh}_{i\hh},X^{\hh}_{(i+1)\hh}\rp>\rho\rb =0 ,
\]
and thus for any $0<T<\infty$,
\[
\lim_{\hh\rightarrow0} \sup_{q\in\bR^N} P^{\hh}_{q}\lb \sup_{\substack{i:0\leq \hh i\leq T \\ 0\leq s\leq \hh}} d_{\bR^N}\lp X^{\hh}_{i\hh} , X^{\hh}_{i\hh+s}\rp \leq \rho\rb =1 .
\]
Thus, we are essentially assuming that the behavior of $X^{\hh}_t$ for $t\in(i\hh,(i+1)\hh)$ is comparable to the behavior of the step $X^{\hh}_{(i+1)i\hh}|X^{\hh}_{i\hh}$, which as already mentioned, is automatic in the Euclidean case but not necessarily in other situations.

Further, we note that the condition of Eq.~\eqref{Eqn:PathRegularity} follows from assuming that $\frac{1}{\hh}\mu^{\hh}_q\lb \|\lambda\| \leq \rho\rb$ approaches 1 as $\hh\rightarrow 0$ uniformly for $q$ in any compact set (here $\lambda$ is a covector in $T_q^*M$). In particular, for random walks of the type we consider earlier in the paper, this condition is trivially satisfied by taking $\hh < \rho^2/2k$, because by construction a step of $X^{\hh}_t$ consists in traveling a geodesic for distance $\sqrt{2  \hh  k}$; that is, $\mu^{\hh}_q$ is supported on covectors of length $\sqrt{2\hh k}$ for all $q$ (again, see Definition \ref{Def:MainOperator}).

\subsection{Convergence results}

We can now prove the following general convergence theorem. Note that the argument is essentially soft, since all of the serious estimates are implicitly already taken care of in \cite{SAndV}.


\begin{theorem}\label{AppendixConvergence}
Let $M$ be a (sub)-Riemannian manifold with a smooth second-order operator $L$ with non-negative-definite principal symbol and without zeroeth-order term. Further, suppose that the diffusion generated by $L$, which we call $X^0$, does not explode, and let $P_q$ be the corresponding probability measure on $\Omega_M$ starting from $q$. Similarly, let $P^{\hh}_q$ be the probability measures on $\Omega_M$ corresponding to a sequence of random walks $X^{\hh}_{t}$ as above with $X^{\hh}_0=q$, and let $L_{\hh}$ be the associated operators. Suppose that, for any $\phi\in C_0^{\infty}(M)$, we have that
\begin{equation}\label{Eqn:OperatorConv}
L_{\hh} \phi\rightarrow L\phi \quad\text{uniformly on compacts as $\hh\rightarrow0$,}
\end{equation}
and also suppose that the condition of Eq.~\eqref{Eqn:PathRegularity} holds for the $X^{\hh}_{t}$.
Then if $q_{\hh}\rightarrow q$ as $\hh\rightarrow0$, we have that $P^{\hh}_{q_{\hh}}\rightarrow P_q$ as $\hh\rightarrow0$.
\end{theorem}
\begin{proof} 
First, suppose that $M$ is compact. Then $M$ can be realized as a compact (isometrically embedded) submanifold of $\bR^N$ via Nash embedding. This makes $\Omega_M$ a closed subset of $\Omega_{\bR^N}$. Further, we claim that there is an increasing, continuous function $u:[0,\infty)\rightarrow[0,\infty)$ with $u(0)=0$ such that, for $q,\tilde{q}\in M$,
\begin{equation}\label{Eqn:DistanceComp}
d_{\bR^N}\lp q,\tilde{q}\rp \leq  d_{M}\lp q,\tilde{q}\rp \leq u\lp d_{\bR^N}\lp q,\tilde{q}\rp \rp .
\end{equation}
In particular, the topology on $\Omega_M$ induced as a subset of $\Omega_{\bR^N}$ agrees with the original topology on $\Omega_M$. To see that this is true, first observe that the inequality $d_{\bR^N}\lp q,\tilde{q}\rp \leq  d_{M}\lp q,\tilde{q}\rp$ is immediate from the isometric embedding and the optimality of Euclidean geodesics in $\bR^N$. Next, we let $\overline{d}_M$ be the Riemannian distance on $M$ (induced by the Riemannian extension metric of the sub-Riemannian metric). Then \cite[Theorem 1.2]{Jea-2014} implies that  every point of $M$ has a neighborhood on which we can find a function $u$ as above for which $d_{M}\lp q,\tilde{q}\rp \leq u\lp \overline{d}_M\lp q,\tilde{q}\rp \rp$. One only needs to check that the constants in \cite[Theorem 1.2]{Jea-2014} can be chosen uniformly, and this follows from the fact that, in \cite[Lemma 1.1]{Jea-2014}, the commutators $X_{I_i}$ can be used on the entire neighborhood (assuming it's small enough) and then the function $\varphi(t_1,\ldots,t_n)$ depends continuously on the base point. Since $M$ is compact, this $u$ can be chosen for all of $M$. Further, since the Riemannian distance $\overline{d}_M$ and the Euclidean distance $d_{\bR^N}$ are Lipschitz-comparable under the embedding, potentially multiplying $u$ by a positive constant allows us to replace  $\overline{d}_M$ by $d_{\bR^N}$. This establishes Eq.~\eqref{Eqn:DistanceComp} (and in fact, $u$ can be taken to be H\"older-continuous, though we don't need that here).

Next, there is no problem extending $L$ to a smooth operator on all of $\bR^N$ with bounded coefficients (where we mean the coefficients used when writing $L$ with respect to the Euclidean coordinates). We also extend the family of random walks $X^{\hh}_t$ to a family of Euclidean random walks $\tilde{X}^{\hh}$ so that we have a random walk starting from any point of $\bR^N$ in such a way that the convergence of Eq.~\eqref{Eqn:OperatorConv} and the condition of Eq.~\eqref{Eqn:PathRegularity} still hold (with  $X^{\hh}_t$ replaced by $\tilde{X}^{\hh}_t$). Here $\tilde{X}^{\hh}_t$ is not exactly an extension, since we interpolate between $\tilde{X}^{\hh}_{i\hh}$ and $\tilde{X}^{\hh}_{(i+1)\hh}$ by a Euclidean geodesic. However, we still have that, if $X^{\hh}_0=\tilde{X}^{\hh}_0\in M$, then $X^{\hh}_{i\hh}=\tilde{X}^{\hh}_{i\hh}$ for all $i$, so that the underlying Markov chains on $M$ agree. Further, it's easy to see that these extensions can be performed in such a way that the assumptions of Theorem 11.2.3 of \cite{SAndV} are satisfied (indeed, this just means some extra global assumptions on the random walks and operators, but since we start from a compact set, we can make $L$ and $\tilde{X}^{\hh}_t$ be anything we'd like outside of some large ball, just by using a bump function).

We let $\tilde{P}^{\hh}$ be the measures on $\Omega_{\bR^N}$ corresponding to the $\tilde{X}^{\hh}$. If we now apply Theorem 11.2.3 of \cite{SAndV} to the situation at hand, see that $\tilde{P}^{\hh}_{q_{\hh}}\rightarrow P_q$ as $\hh\rightarrow0$ (as probability measures on $\Omega_{\bR^N}$). It remains only to see that the same holds for $P^{\hh}_{q_{\hh}}$, that is, that whether or not we go from $X_{i\hh}$ to $X_{(i+1)\hh}$ via a Euclidean geodesic or via the original $M$-geodesic doesn't matter in the limit.

We now pass to any discrete subsequence $\hh_j$ such that $\hh_j\rightarrow 0$ as $j\rightarrow \infty$. However, for simplicity, we continue just to write $\hh$. By Theorem 1.3.1 of \cite{SAndV} and the fact that $\{\tilde{P}^{\hh}_{q_{\hh}}\}$ is precompact (indeed, it converges as $\hh\rightarrow 0$), we see that, for every $\rho>0$ and $0<T<\infty$,
\[
\lim_{\delta\searrow 0} \inf_{\hh} \tilde{P}^{\hh}_{q_{\hh}}\lb \sup_{\substack{0\leq s\leq t\leq T \\ t-s\leq\delta}}  d_{\bR^N}\lp \tilde{X}^{\hh}_t,\tilde{X}^{\hh}_s\rp\leq \rho \rb =1 .
\]
(Here the $\hh$ on both the measure $P^{\hh}$ and the corresponding random path $\tilde{X}^{\hh}$ is somewhat redundant, but clear nonetheless. In a moment we will make better use of this notation.) 

By assumption, the condition of Eq.~\eqref{Eqn:PathRegularity} holds. Since $M$ is compact and $X^{\hh}_{i\hh}$ is a Markov chain, simply summing the probability of the distance exceeding $\rho$ at each step of the walk gives that for any $0<T<\infty$, any $\rho>0$, and any $\alpha>0$, there exists $\hh_0>0$ such that
\begin{equation}\label{Eqn:PathRegularity2}
\sup_{q\in M}P^{\hh}_{q}\lb \sup_{\substack{i : 0\leq \hh i\leq T \\ s<\hh}} d_M\lp X^{\hh}_{i\hh} , X^{\hh}_{i\hh+s}\rp > \rho\rb \leq \frac{T+1}{\hh} \sup_{q\in M} P^{\hh}_{q}\lb \sup_{0\leq s\leq\hh} d_M\lp X^{\hh}_{0} , X^{\hh}_{s}\rp > \rho\rb \leq (T+1)\alpha
\end{equation}
whenever $\hh<\hh_0$. By Eq.~\eqref{Eqn:DistanceComp}, this also holds if we replace $d_M$ with $d_{\bR^N}$, perhaps by taking $\hh$ smaller. Thus, for any $0<T<\infty$, any $\rho>0$, any $\alpha>0$ and small enough $\hh$, we have that
\[
P^{\hh}_{q_{\hh}}\lb \sup_{\substack{i : 0\leq \hh i\leq T \\ s<\hh}} d_{\bR^N}\lp X^{\hh}_{i\hh} , X^{\hh}_{i\hh+s}\rp \leq \rho\rb > 1-\alpha.
\]

In order to compare $\tilde{P}^{\hh}_{q_{\hh}}$ and $P^{\hh}_{q_{\hh}}$, it is convenient to realize them as push-forwards of a single measure on some probability space. In some sense, this is how we have described these random walks as being generated, in terms of a sequence of independent random variables which we use to draw the new cotangent vector at every step. However, it is more direct just to note that the paths $\tilde{X}_t$ can be recovered from the paths $X_t$. More concretely, for any path $\omega$ on $M$ that is piecewise geodesic with respect to the times $0,\hh,2\hh,\ldots$, let $F^{\hh}(\omega)$ be the path in $\bR^N$ that interpolates between $\omega_{i\hh}$ and $\omega_{(i+1)\hh}$ by the appropriate Euclidean geodesic for each $i$. Then $F^{\hh}$ is defined on the support of $P^{\hh}_{q_{\hh}}$, and $\tilde{P}^{\hh}_{q_{\hh}}$ is the pushforward of $P^{\hh}_{q_{\hh}}$ under $F^{\hh}$ for each $\hh$. So it is natural to view both $X^{\hh}$ and $\tilde{X}^{\hh}$ as random variables under $P^{\hh}_{q_{\hh}}$, with $\tilde{X}^{\hh}= F^{\hh}\lp X^{\hh}\rp$.

It now follows from the above (recall in particular that  $\tilde{P}^{\hh}_{q_{\hh}}$ and $P^{\hh}_{q_{\hh}}$ have the same marginals on the sequence of times $0, \hh, 2\hh,\ldots$) that, given $0<T<\infty$, $\rho>0$, and $\alpha>0$, for all sufficiently small $\hh$ we have
\[
P^{\hh}_{q_{\hh}}\lb \sup_{0\leq t \leq T} d_{\bR^N}\lp \tilde{X}^{\hh}_{t} , X^{\hh}_{t}\rp \leq \rho\rb > 1-\alpha .
\]
In light of Eq.~\eqref{Eqn:PathspaceDist}, we conclude that, for any $i\in\{1,2,\ldots\}$, $\rho>0$, any $\alpha>0$, and sufficiently large $\hh$ we have:
\[
P^{\hh}_{q_{\hh}}\lb d_{\bR^N}\lp \tilde{X}^{\hh}, X^{\hh}\rp \leq \rho + \frac{1}{2^i} \rb > 1-\alpha.
\]

Next, let $\Phi$ be a bounded, uniformly continuous function on $\Omega_{\bR^N}$, and choose any $\delta>0$. Then there exists $\eta>0$ such that $\lab\Phi(\omega)-\Phi(\tilde{\omega})\rab<\delta$ whenever $d_{\bR^N}(\omega,\tilde{\omega})<\eta$. Now choose $i$ large enough so that $\sum_{j=i}^{\infty}1/2^j < \eta/2$ and let $\rho<\eta/2$ and $\alpha<\delta$ in the above. Then there exists $\hh_0>0$ such that
\[
P^{\hh}_{q_{\hh}}\lb d_{\bR^N}\lp \tilde{X}^{\hh}, X^{\hh}\rp > \eta \rb < \delta ,
\]
and thus
\[
\E^{P^{\hh}_{q_{\hh}}}\lb \lab \Phi\lp \tilde{X}^{\hh}\rp-\Phi\lp X^{\hh}\rp\rab \rb  <\delta + 2\delta \|\Phi\|_{\infty}
\]
for $\hh<\hh_0$ (here $ \|\Phi\|_{\infty}$ is the $L^{\infty}$-norm of $\Phi$). Further, using $\tilde{P}^{\hh}_{q_{\hh}}\rightarrow P_q$, after perhaps decreasing $\hh_0$, we have
\[
\lab \E^{P^{\hh}_{q_{\hh}}}\lb \Phi\lp \tilde{X}^{\hh}\rp\rb - \E^{P_{q}}\lb \Phi\lp X^{0}\rp\rb\rab<\delta
\]
for $\hh<\hh_0$. Then linearity of expectation and the triangle inequality imply that
\[
\lab \E^{P^{\hh}_{q_{\hh}}}\lb \Phi\lp X^{\hh}\rp\rb - \E^{P_{q}}\lb \Phi\lp X^{0}\rp\rb\rab<2\delta\lp 1+ \|\Phi\|_{\infty}\rp
\]
for $\hh<\hh_0$, and thus, since $\delta>0$ is arbitrary
\[
\lim_{\hh\rightarrow 0} \E^{P^{\hh}_{q_{\hh}}}\lb \Phi\lp X^{\hh}\rp\rb = \E^{P_{q}}\lb \Phi\lp X^{0}\rp\rb .
\]
This convergence holds for any bounded, uniformly continuous function $\Phi$ on $\Omega_{\bR^N}$, and this is sufficient (by the portemanteau theorem), to show that $P^{\hh}_{q_{\hh}}\rightarrow P_q$.  Here there is some ambiguity as to whether or not we think of $P^{\hh}_{q_{\hh}}$ and $P_q$ as probability measures on $\Omega_M$ or $\Omega_{\bR^N}$, but the point is that it doesn't matter. As measures on $\Omega_{\bR^N}$, they are nonetheless supported on $\Omega_M$. Further, weak convergence over $\Omega_{\bR^N}$ implies weak convergence over $\Omega_M$, and thus the convergence pulls back to $\Omega_M$. Since this convergence holds for any discrete sequence $\hh_j\rightarrow 0$, it holds as $\hh\rightarrow 0$ via all positive reals, and this completes the proof in the case when $M$ is compact.

If $M$ is not compact, we proceed by exhaustion, using Lemma 11.1.1 of \cite{SAndV}. In this context, note that this lemma is stated for the case when $\Omega$ is $\Omega_{\bR^N}$, but it is proved by a general and elementary method, and thus the statement and proof apply, as written, to the case when $\Omega=\Omega_M$.

Let $A_k\subset M$ be an exhaustion of $M$ by compact subsets (with smooth boundary, possible by Sard's theorem). We assume that $q$ and all of the $q_{\hh}$ are contained in $A_1$. If $\tau_k$ is the first hitting time of the complement of the interior of $A_k$, then $\tau_k$ is a non-decreasing sequence of lower semi-continuous stopping times that increases to $\infty$ for each $\omega\in\Omega_M$. Now let $M_k$ be a sequence of compact Riemannian manifolds into which $A_k$ can be isometrically included. (In other words, we truncate, in a geometrically reasonable way, $M$ outside of $A_k$ to get $M_k$.) For each $k$, extend $L$ to a smooth operator $L_k$ on $M_k$, and extend the random walk to $M_k$ in such a way that the probability measures $P^{\hh, k}_{q_{\hh}}$ corresponding to the random walks on $M_k$ converge to the the probability measure corresponding to the diffusion generated by $L_k$, which we denote $Q^k_{q}$. The previous argument for compact $M$ makes it clear that this is possible.

By construction, the probability measure $P^{\hh, k}_{q_{\hh}}$ agrees with $P^{\hh}_{q_\hh}$ on $\sM_{\tau_k}$. Also, by the previous argument, for any $k\geq 1$ and any discrete sequence $\hh_j\rightarrow 0$, the family $\lc P^{\hh_j, k}_{q_{{\hh_j}}} :j\geq 1 \rc$ converges to $Q^k_q$, and $Q^k_q$ equals $P_q$ on $\sM_{\tau_k}$. These are exactly the assumptions of Lemma 11.1.1 of \cite{SAndV}, and so we conclude that $P^{\hh_j}_{q_{\hh_j}}\rightarrow P_q$ as $j\rightarrow \infty$. As before, since this holds for any discrete sequence $\hh_j\rightarrow 0$ we have that $P^{\hh}_{q_\hh}\rightarrow P_q$ as $\hh\rightarrow 0$.
\end{proof}

\begin{rmk}
Note that nowhere in the proof did we use that the path from $X^{\hh}_{\hh i}$ to $X^{\hh}_{\hh(i+1)}$ was a geodesic. Indeed, all that matters is that the path is continuous and satisfies the condition of Eq.~\eqref{Eqn:PathRegularity} (and goes from $X^{\hh}_{\hh i}$ to $X^{\hh}_{\hh(i+1)}$, of course). Thus, Theorem \ref{AppendixConvergence} holds in slightly more generality, where we extend the class of random walks under consideration to include random walks where the interpolations from $X^{\hh}_{\hh i}$ to $X^{\hh}_{\hh(i+1)}$ are not necessarily geodesics. Perhaps the most natural such situation would be where each step of the random walk is given by flowing along the integral curve of some horizontal vector field for a small time (as opposed to traveling along a geodesic for a small time), as considered in \cite{OurVolume}.
\end{rmk}

Finally, we return to the special case of the type of random walks discussed earlier in the paper. Then we can show that the convergence of a single step of the random walk, used to (first) define the microscopic Laplacian with respect to a splitting, also implies the convergence of the random walk to the natural limiting diffusion.

\begin{theorem}\label{AppendixConvergence2}
For a (sub)-Riemannian manifold $M$, consider a splitting $TM =\distr \oplus \V$ (if $M$ is Riemannian, the splitting is necessarily trivial) and some choice of adapted measure $\{\mu_q\}_{q \in M}$, and let $L^\V$ be the associated microscopic Laplacian, as in Definition \ref{Def:MainOperator}. Then there is a unique diffusion $X^{0,q}_{t}$ generated by $L^\V$ starting from any $q$. Let $X^{\hh,q}_{t}$ be the random walk, starting from $q$, determined by the parabolic scaling of the family $\{\mu_q\}$, as discussed above.  Assume that $X^{0,q}_{t}$ does not explode, for any $q$. If $q_{\hh}\rightarrow q_0$ as $\hh\rightarrow 0$, then the random walks $X^{\hh,q_{\hh}}_t$ converge to the diffusion $X^{0,q_0}_t$ as $\hh\rightarrow 0$, in the sense that the corresponding probability measures on $\Omega_M$ converge weakly.
\end{theorem}
\begin{proof}
By Theorem \ref{t:microformula}, the operator $L^\V$ satisfies the assumptions of Theorem \ref{AppendixConvergence}. Also, the existence and uniqueness of $X^{0,q}_t$ follows. Now let $L_{\hh}$ be the operator associated to the random walks $X^{\hh,\cdot}_t$, as described above. Then for any $\phi\in C_0^{\infty}(M)$, we see that $L_{\hh} \phi\rightarrow L^\V \phi$ uniformly (on all of $M$) as $\hh\rightarrow0$, by Theorem \ref{t:microformula} and the fact that $L_{\hh} \phi$, for all small $\hh$, has support in a fixed compact set by construction. Finally, we have already noted that  the condition of Eq.~\eqref{Eqn:PathRegularity} holds trivially for the random walks $X^{\hh,q}_{t}$. Thus we can apply Theorem \ref{AppendixConvergence}, which gives the desired convergence of randoms walks.
\end{proof}


\subsection*{Acknowledgments}

The authors are grateful to Fr\'ed\'eric Jean for his advice on non-equiregular structures, and in particular, for explaining how to derive Eq.~\eqref{Eqn:DistanceComp}. We also thank Andrei Agrachev, Gr\'egoire Charlot, Erlend Grong, and Jose Veloso for useful and stimulating discussions. We finally thank the Institut Henri Poincar\'e (Paris) for the hospitality during the Trimester ``Geometry, Analysis and Dynamics on sub-Riemannian manifolds,'' where most of this research has been carried out.

This research has been partially supported by the European Research Council, ERC StG 2009 ``GeCoMethods'', contract n.\ 239748, by the ERC POC project ARTIV1 contract number 727283, by the ANR project ``SRGI'' ANR-15-CE40-0018, by a public grant as part of the Investissement d'avenir project, reference ANR-11-LABX-0056-LMH, LabEx LMH (in a joint call with Programme Gaspard Monge en Optimisation et Recherche Op\'erationnelle), by the iCODE institute, research project of the Idex Paris-Saclay, and by the SMAI project ``BOUM.''

\bibliographystyle{abbrv}
\bibliography{bibliorandom}

\end{document}